\documentclass[12pt]{article}
\usepackage{graphicx}
\usepackage{amssymb,amsmath,amsthm}
\usepackage{bbold}
\usepackage{enumitem}

\DeclareMathAlphabet{\mathpzc}{OT1}{pzc}{m}{it}

\def\simple{plain}

\def\K{{\Bbb K}}

\def\bN{\mathbb{N}}

\def\bC{\mathbb{C}}
\def\bR{\mathbb{R}}

\def\RPo{{\mathbb{RP}^1}}
\def\RPt{{\mathbb{RP}^2}}
\def\KPt{{\mathbb{KP}^2}}
\def\KPn{{\mathbb{KP}^n}}
\def\CPo{{\mathbb{CP}^1}}
\def\Rt{{\mathbb{R}^2}}

\def\cA{{\cal A}}

\def\cB{{\cal B}}

\def\cE{{\cal E}}

\def\cI{{\cal I}}

\def\cK{{\cal K}}
\def\cH{{\cal H}}

\def\cF{{\cal F}}

\newtheorem{definition}{Definition}
\newtheorem{theorem}{Theorem}

\newtheorem{corollary}{Corollary}
\newtheorem{example}{Example}

\newtheorem{remark}{Remark}

\newtheorem{proposition}{Proposition}
\newtheorem{conjecture}{Conjecture}
\usepackage{array}
\newcolumntype{S}{>{\centering\arraybackslash} m{.475\linewidth}}
\newcolumntype{T}{>{\arraybackslash} m{10.5cm}}
\newcolumntype{U}{>{\centering\arraybackslash} m{2cm}}
\title{%
  ``Simple Dynamics'' conjectures\\ for some real Newton maps on $\bR^2$.
}
\author{Roberto De Leo\\ \small Howard University, Washington DC 20059 (USA)}
\begin{document}
\maketitle
\begin{abstract}
  We collect from several sources some of the most important results on the forward and backward limits of points under real and complex rational functions,
  and in particular real and complex Newton maps, and we provide numerical evidence that a fundamental result by B.~Barna on the dynamics of Newton's
  method on $\bR$ extends to $\bR^2$.
\end{abstract}
\section{Introduction}
One of the most natural ways to understand the behavior of a continuous surjective map $f$ of a compact manifold $M$ into itself is studying the asymptotics
of the forward and backward orbits of the points of $M$ under $f$. Among the simplest things that can happen is that there is some finite number of
attracting fixed points $c_i$ %(attractors)
%closed disjoint forward-invariant subsets $A_i$ (attractors)
such that the sequence of iterates $\{f^n(x)\}$ converges to one of them for almost all $x\in M$ (with respect to any measure equivalent to the Lebesgue
measure on each chart) and that, again for some full measure set,  the sets $f^{-n}(x)$ converge, in some suitable sense, to the set of points whose forward
iterates do not converge to any $c_i$.
%some finite number of closed disjoint backward-invariant subsets $R_j$ (repellors) such that
%the sets $f^{-n}(x)$ converge, in some suitable sense, to one of them for almost all $x\in M$.
In other words, the action of $f$ on $M$ is, asymptotically, to thicken points near the $c_i$ while, at the same time, thinning them out near the
boundaries of the basins of attraction of the $c_i$. We call functions with such behavior {\em \simple}.

While the large diversity and complexity of behaviors of continuous maps of a manifold $M$ into itself suggests that in the general case the situation
is much more complicated, it was a surprising discovery that the same is true even in case of very elementary maps such as quadratic polynomials in one
variable (e.g. the logistic map) or piecewise linear polynomials in one variable (e.g. the tent map) -- see~\cite{JR80,Avi05,Lyu12,vSt10} and the references
therein for a large panorama of the old and recent advances in this field.

Since being \simple\ does not seem frequent among continuous functions, it is particularly important singling out properties that identify families
of function that behave so nicely under iteration.
A large source of them is given by the rational maps coming from complex Newton's method. Consider, for instance, the case of the complex polynomial $f(z)=z^3-1$,
whose Newton map $N_f:\CPo\to\CPo$ is given by $N_f(z)=\frac{2z^3+1}{3z^2}$. It is well known that $N_f$ has exactly three attractors, the cubic roots of the unity, and one repellor,
namely the Julia set of $N_f$ (see Fig.~\ref{fig:cjs}), which means that $N_f$ is \simple.
Note that the situation can get more complicated even with different polynomials of same degree:
as it was shown first numerically by D. Sullivan et al.~\cite{CGS83}, in the space of all complex cubic polynomials there is a set of non-zero Lebesgue
measure for which there exist attracting cycles and, for each of these polynomials, the basin of attraction of such attracting cycle has measure larger than zero.

While the dynamics of Newton maps on the complex line has been deeply and thoroughly studied over the last 40 years, especially in connection
with the general problem of the dynamics of complex rational maps in one variable initiated exactly 100 years ago by G. Julia~\cite{Jul18}
and P. Fatou~\cite{Fat19,Fat20a,Fat20b}, in comparison almost nothing has been done in the more general problem of Newton maps on the real plane.
{\bf The main aim of  this article is to attract the attention of the dynamical systems community on this topic by providing numerical
  evidence that Newton's maps on the real plane relative to generic polynomials with only real roots are \simple.}
%
%\section{Definitions, notations and review of main properties of Newton maps on the real and complex lines.}
\section{Preliminaries}
The following concepts are central for the present article:
\begin{definition}
  Let $M$ be a compact manifold with a measure $\mu$ and $f$ a surjective continuous map of a manifold $M$ into itself.
  %We denote by $f^n$ the $n$^{th} iterate of $f$
  The $\omega$-limit of a $x\in M$ under $f$ is the (closed) set of all points to whom $x$ iterate are eventually close, namely
  %accumulation points of the sequence $\{f^n(x)\}$, namely
  $$
  \omega_f(x) = \bigcap_{n\geq0}\;\overline{\bigcup_{m\geq n}\{f^m(x)\}}\,,
  $$
  while its $\alpha$-limit is the (closed) set of points to whom $x$ iterated counterimages are eventually close, namely
  %accumulations points of the sequence $\{f^{-n}(x)\}$, namely
  $$
  \alpha_f(x) = \bigcap_{n\geq0}\;\overline{\bigcup_{m\geq n}\{f^{-m}(x)\}}\,.
  $$
  The $\omega$ and $\alpha$ limits of a set is the union of the $\omega$ and $\alpha$ limits of all of its points.
  The {\em forward (\hbox{\sl resp.} backward) basin} $\cF_f(C)$ (resp. $\cB_f(C)$) under $f$ of a closed invariant subset $C\subset M$ is the set
  of all $x\in M$ such that $\omega_f(x)\subset C$ (resp. $\alpha_f(x)\subset C$).
  Following Milnor~\cite{Mil85}, we say that a closed subset $C\subset M$ is an {\em attractor} (resp. {\em repellor}) for $f$ if:
  \begin{enumerate}
    \item $\cF_f(C)$ (resp. $\cB_f(C)$) has strictly positive measure;
    \item there is no closed subset $C'\subset C$ such that $\cF_f(C)$ (resp. $\cB_f(C)$) coincides with $\cF_f(C')$ (resp. $\cB_f(C')$) up to a null set.
  \end{enumerate}
  Finally, we say that $f$ is {\em\simple} if it has a finite number of attracting fixed points $c_i$, $i=1,\dots,N$, so that:
  %and repellors $R_j$ and the following holds:
  \begin{enumerate}[label=\roman*.]
  \item $\cup_{i=1}^N\cF_f(c_i)=M\setminus  J$ is a full measure set;
  \item $\cB_f(J)$ is a full measure set.
  \end{enumerate}
  If (ii) holds at least for a non-empty open set, then we say that $f$ is {\em weakly \simple}.
  %  both the complement of the attractors and the retractors are sets of zero measure with respect to $\mu$.
\end{definition}
\begin{remark}
  Conditions (i) and (ii) imply that a \simple\ map cannot have any other attractor/repellor besides the $c_i$ and $J$.
\end{remark}
\begin{remark}
  While $\omega$-limits of discrete systems have been thoroughly studied, at least in one (real and complex) dimension,
  relatively very little has been done to date for $\alpha$-limits (see~\cite{BGL13} for a discussion on this topic).
\end{remark}
Finite attractors and repellors play a major role in this theory:
\begin{definition}
  A {\em periodic orbit} (or {\em $k$-cycle}) $\gamma$ is a non-empty finite set of $k$ points minimally invariant under $f$, namely
  that cannot be decomposed into the disjoint union of smaller invariant sets.
  A 1-cycle is also called 
%  When $gamma=\{z_0\}$, we simply say that $z_0$ is
  a {\em fixed point}.
\end{definition}
\begin{example}
  Consider the map $f$ of the Riemann sphere $\CPo$ into itself given by $f([z:w])=[2z:w]$. The only invariant proper sets of $f$
  are the two fixed points, the {\em south pole} $S=[0:1]$ (repellor) and the {\em north pole} $N=[1:0]$ (attractor). Clearly $\omega_f(x)=N$
  for all $x$ but $S$ and $\alpha_f(x)=S$ for all $x$ but $N$, so that $\cF_f(N)=\CPo\setminus\{S\}$ and $\cB_f(S)=\CPo\setminus\{N\}$. 
  In particular, $f$ is \simple.
\end{example}
The following sets are of fundamental importance in the dynamics of a continuous map:
\begin{definition}
  Given a compact manifold $M$ and a continuous map $f:M\to M$, the {\em Fatou set} $F_f\subset M$ of $f$ is the largest open set
  over which the family of iterates $\{f^n\}$ is {\em normal}, namely the largest open set over which there is a subsequence of the iterates of
  $f$ that converges locally uniformly. The complement of $F_f$ in $M$ is the {\em Julia set} $J_f$ of $f$.
  Finally, we denote by $Z_f$ the set of points $x\in M$ where the Jacobian $D_xf$ is degenerate.
\end{definition}
In the present article we focus on the case of rational functions so, from now on, we will restrict to this case all our definitions and statements.
%
%\begin{definition}
%  Given a rational map $f$ of the Riemann sphere $\CPo$ into itself, the {\em Fatou set} $F_f\subset\CPo$ of $f$ is the largest open set
%  over which the family $\{f^n\}$ is normal. The complement of $F_f$ in $\CPo$ is the {\em Julia set} $J_f$ of $f$.
%\end{definition}
%
\begin{theorem}[Julia, Fatou, 1918]
  \label{thm:JF}
  Let $f$ be a rational complex map of degree larger than 1. Then:
  \begin{enumerate}
  \item $F_f$ contains all basins of attractions of $f$;
  \item both $J_f$ and $F_f$ are forward and backward invariant and $J_f$ is the smallest
    closed set with more than 2 points with such property;
%  \item $J_f$ is the smallest forward and backward closed set with more than 2 points;
  \item $J_f$ is a perfect set;
  \item $J_f$ has interior points iff $F_f=\emptyset$; % (e.g. $F_f=\emptyset$ for $f(x)=\left((z-2)/z\right)^2$, see Corollary 6.2.4 in~\cite{HP88});
  \item $J_f=J_{f^n}$ for all $n\in\bN$;
  \item $J_f$ is the closure of all repelling cycles of $f$;
  \item there is an open dense $U\subset J_f$ s.t. $\omega_f(z)=J_f$ for all $z\in U$;
  \item for every $y\in J_f$, $J_f=\overline{\{x\in\CPo\,|\;f^k(x)=y\hbox{ for some $k\in\bN$}\}}$;
  \item for any attracting periodic orbit $\gamma$, we have that $\partial\cF_f(\gamma)=J_f$.
  \item the dynamics of the restriction of $f$ to its Julia set is highly sensitive to the initial conditions, namely $f|_{J_f}$ is {\em chaotic}.
  \end{enumerate}
\end{theorem}
\begin{remark}
  Both possibilities of point 4 above take place. Two well-known examples of functions with empty Fatou set are
  the Latt\`es example~\cite{Lat18} $p(z)=\frac{(z^2+1)^2}{4z(z^2-1)}$, related to the theory of Elliptic functions
  (see~\cite{Bea00}), and $q(z)=\frac{(z-2)^2}{z^2}$ (see~\cite{Bea00} and Corollary 6.2.4 in~\cite{HP88}).
  In general, $F_f=\emptyset$ when $\omega_f(z)=\CPo$ for some $z$ (see~\cite{Bea00}, Thm~4.3.2).
\end{remark}
\begin{remark}
  As soon as $f$ has more than 2 attractive fixed points, $J_f$ must have a fractal nature since, by point 9 above, each of its point
  belongs to the boundary of more than 2 basins of attraction. In other words, in this case all basins of attraction have the Wada
  property~\cite{KY91} (see~\cite{PR86} and Sec.~4.1 of~\cite{Mil06} for a series of examples and pictures of fractal Julia sets of
  polynomial, rational and transcendental complex maps).
\end{remark}
\begin{example}
  \label{ex:z2}
  Consider the map $f(z)=z^2$ and denote by $E$, $L$ and $U$ respectively the {\em equator} $\{|z|=1\}$, the {\em lower hemisphere}
  $\{|z|<1\}$ and the {\em upper hemisphere} $\{|z|>1\}\cup\{N\}$.

  If $z\in L$ (resp. $U$), then $\{f^n\}$ converges uniformly in some neighborhood of $z$ to the constant map $z\mapsto S$
  (resp. $z\mapsto N$), so $L=\cF_f(S)\subset F_f$ (resp. $U=\cF_f(N)\subset F_f$), namely $f$ has exactly two attractors
  (the south and north poles) and $F_f$ is the disjoint union of their basins.
  On the contrary, if $z\in E$, then for any neighborhood of $z$ there will be some point converging to $N$ and some point converging to $S$
  under $\{f^n\}$, so the family is not normal and $J_f=E$.  
%  Hence in this case $F_f=L\sqcup U$ and $J_f=E$

  Notice that, as claimed by the theorem above, $\partial\cF_f(N)=\partial\cF_f(S)=J_f$ and that $f|_{J_f}$ is the doubling map
  on the circle, a classic example of chaotic map. Finally, notice that $J_f$ is also the only repellor of $f$ and that its basin is given by
  $\cB_f(J_f)=\CPo\setminus\{S,N\}$. In particular, $f$ is \simple.
\end{example}
Notice that the north and south poles in the example above both have a finite $\alpha$-limit set: the only point in their counterimages
is themselves. This exceptional behavior can happen at most at two points for any complex rational map of degree two or more (see point 2 above
and~\cite{Bea00}, Thm~4.2.2). The fact that the $\alpha$-limit of every other point is $J_f$ it is a general result very useful in numerical
exploration of Julia sets:
%
%\begin{theorem}[see~\cite{Bea00}, Thm 4.2.7 and~\cite{Bro65}, Thm 6.1}]
%\begin{theorem}[see~\cite{Bro65}, Thm 6.1]
%\begin{theorem}[Cremer, 1932~\cite{Cre32} (see also~\cite{Bro65}, Thm 6.1)]
\begin{theorem}[Fatou, 1920~\cite{Fat20a} (see also~\cite{Bro65}, Thm 6.1 and Lemma 6.3)]
  \label{thm:Jf}
    Let $f$ be a rational complex map of degree larger than 1. Then for all $z\in\CPo$, with at most two exceptions,
    $\alpha_f(z)\supset J_f$. Moreover, $\alpha_f(z)=J_f$ if and only if $z$ belongs to either $J_f$ or to the basin
    of attraction of a root of $f$ (except the root itself if it does not belong to $J_f$).
    More generally, if $E\subset\CPo$ is a closed set disjoint from $\omega_f(F_f)$, then the sequence of sets $E_n=f^{-n}(E)$ converges
    uniformly to $J_f$.
    %    Let $z$ belong to either $J_f$
%    For all $z\in\CPo$, with at most two exceptions,
%    Then $\alpha_f(z)=J_f$.
\end{theorem}
\begin{corollary}
  Let $f$ be a rational complex map of degree larger than 1 whose only attractors are its fixed points and whose Julia set has measure zero.
  Then $f$ is \simple.
\end{corollary}
Theorem~\ref{thm:Jf} is reminiscent of what happens in case of hyperbolic Iterated Function Systems:
\begin{definition}
  A Iterated Function System (IFS) $\cI$ on a metric space $(X,d)$ is a semigroup generated by some finite number of continuous functions
  $f_i:X\to X$, $i=1,\dots,n$. We say that $\cI$ is {\em hyperbolic} when the $f_i$ are all contractions. The Hutchinson operator associated to
  $\cI$ is defined as $\cH(A)=\cup_{i=1}^n f_i(A)$, $A\subset X$.
\end{definition}
\begin{theorem}[Hutchinson~\cite{Hut81}, 1981; Barnsley and Demko~\cite{BD85}, 1985]
  Let $\cI$ be a hyperbolic IFS on $X$. Then there exists a unique non-empty compact set $K\subset X$ such that $\cH(K)=K$. Moreover,
  $\lim_{n\to\infty}\cH^n(A)=K$ for every non-empty compact set $A\subset X$.
\end{theorem}
In the simplest cases, like Example~\ref{ex:z2} above, the naive idea is that ultimately the map $f$ is a contraction close to its attractors while
its inverses $\{w_1,\dots,w_d\}$ (in case of complex rational maps, as many as their degree) are contractions close to its repellor
(in the example above, the Julia set of $f$), which suggests that {\bf the Julia set can be found as the unique invariant compact set of the
  IFS defined by the $\boldsymbol{w_i}$}.
Indeed, in Section 7.3 of~\cite{Bar88}, Barnsley shows through an example how to apply these ideas to Julia sets of rational maps,
namely how to write a Julia set as the invariant compact set of an IFS (notice that this important point of view is seldom
mentioned in the literature about the dynamics of complex rational maps). In the real case the number of counterimages,
even taking into account multiplicity, is not the same for every point and it seems unlikely to be able to build in general an IFS
out of them but, nevertheless, Barnsley's result shows that the invariant set for an open map $f$ under mild conditions can be obtained
as the limit of its inverse images. 

%From the two theorems above it is clear that the IFS generated by the $n$ functions $w_i$, corresponding to the $n$ branches of the inverse
%of a complex rational function $f$ of degree $n$, behave similarly to those of a hyperbolic Iterated Function System. 
%In fact, although IFS associated to holomorphic dynamics are not in general hyperbolic, in~\cite{Bar88} showed that, modulo a suitable
%restriction of the $w_i$ to some subset of $\RPt$, the theorem above holds nevertheless for them too.

A similar result, weaker but much more general, was stated by M. Barnsley (see~Sec.~7.4 of~\cite{Bar88}) in the setting
of continuous maps between metric spaces:
%In Section 7.4 of~\cite{Bar88}, M. Barnsley formalizes and generalizes this idea providing a simple very general sufficient condition for the existence
%of a repellor of a continuous map:
%
\begin{theorem}[Barnsley, 1988]
  \label{thm:Bar}
  Let $(Y,d)$ be a complete metric space and $X$ a compact non-empty proper subset of $Y$.
  Denote by $\cK(X)$ the set of the non-empty compact subsets of $X$ endowed with the Hausdorff distance $h$
  (recall that $h$ makes $\cK(X)$ a complete metric space).
  Assume that one of the following conditions is satisfied:
%  $f$ is either a continuous open map $X\to Y$ such that  $f(X)\supset X$ or a continuous open map $Y\to Y$
  \begin{enumerate}
    \item $f:X\to Y$ is an open map such that $f(X)\supset X$;
    \item $f:Y\to Y$ is an open map such that $f(X)\supset X$ and $f^{-1}(X)\subset X$.
  \end{enumerate}
%  such that both $f(X)\supset X$ and $f^{-1}(X)\subset X$,
  Then the map $F:\cK(X)\to\cK(X)$ defined by $F(K) = f^{-1}(K)$ is continuous, $\{F^n(K)\}$ is a Cauchy sequence, its
  %(closed, compact, invariant)
  limit $K_0=\lim F^n(X)\in\cK(X)$ is a repellor for $f$ and it is equal to the set of points that never leave $X$ under
  the action of $f$.
\end{theorem}

A useful algorithm based on these ideas was extracted by J. Hawkins and M. Taylor~\cite{HT03}
from a Barnsley algorithm introduced in~\cite{Bar88} for certain types of hyperbolic rational maps.
\begin{definition}
  Let $f$ be a rational map of degree $d$. Given a point $z_0\in\CPo$, we call {\em backward orbit} of $z_0$
  any sequence $\zeta_{z_0}=\{z_i\}_{i\in\bN}$ such that $f(z_i)=z_{i-1}$ for all $i$. We endow the space of all
  backward orbits of $z_0$ with the equidistributed Bernoulli  measure $\nu$, namely the measure of the
  set of all sequences $\zeta_{z_0}$ with first $k$ elements $\{z_0,z_1,\dots,z_k\}$ is $d^{-k}$.
\end{definition}
Based on a fundamental result of Freire, Lopes and Ma\~n\'e~\cite{FLM83}, Ma\~n\'e~\cite{Man83}
and, independently, Lyubich~\cite{Lyu83}, Hawkins and Taylor were able to prove the following:
\begin{theorem}[Hawkins, Taylor (2003)]
  Let $f$ be a rational map of degree at least 2 and $z_0$ a non-exceptional point. Then,
  for $\nu$-almost all backward paths $\zeta_{z_0}$, the set of limit points of $\zeta_{z_0}$
  is equal to $J_f$.
%  $$
%  J_f\subset\overline{\cup_{i=0}^\infty\{z_i\}}
%  $$
\end{theorem}
%
%A full proof by J. Hawkins and M. Taylor, showing that Barnsley's idea works indeed for all
%{\em complex} rational functions can be found in~\cite{HT03}.

%above shows that, under mild conditions on a map ,
%it can be proven the existence of a non-empty invariant set built as the limit of the inverse images of 
%The following theorem by  M. Barnsley (see Section 7.4 of~\cite{Bar88}), that does not assume
%anything about the number of counterimages, is therefore particularly helpful:
%In~\cite{Bar88}, M. Barnsley provides simple condition for the uniqueness of a repellor of a continuous map:

Going back to {\em forward} dynamics, since continuous functions maps preserve connectedness, every
complex rational map $f$ induces a mapping of the connected components of $F_f$ in themselves
whose dynamics tells us what happens to the points of the Fatou set under iterations: 
%that ultimately every component falls into four possible types:
%Another important aspect of the dynamics of a complex rational map is that $f$ induces a mapping of the connected
%components of $F_f$ in itself and ultimately every component falls into four possible types:% whose dynamics plays an important role:
%
\begin{theorem}[Sullivan~\cite{Sul85}]
  Let $f$ be a rational complex map of degree larger than 1. Then every connected component of $F_f$ ends up in a finite time
  inside a connected component $V$ of one of the following types:
  \begin{enumerate}
  \item an {\em attractive basin}, namely $V$ contains an attracting periodic point $z_0$ of some period
    $N\geq1$ such that $\lim_{n\to\infty}f^{nN}(z)=z_0$ for all $z\in V$;
  \item a {\em parabolic basin}, namely $\partial V$ contains an attracting periodic point $z_0$ of some period
    $N\geq1$ such that $\lim_{n\to\infty}f^{nN}(z)=z_0$ for all $z\in V$;
  \item a {\em Siegel disc}, namely $f|_V$ is conformally conjugated to an irrational rotation of the unit disc;
  \item a {\em Arnold-Herman ring}, namely $f|_V$ is conformally conjugated to an irrational rotation of an annulus of finite modulus. 
  \end{enumerate}
\end{theorem}
Note that, in case of general entire maps, there might be countably many disjoint connected components $W_n$
of the Fatou set such that $f(W_n)\subset W_{n+1}$. Such sets are called {\em wandering domains} and the
fact that they cannot arise for rational maps is one of the most important contents of Sullivan's theorem above,
often called {\em Non Wandering Domain} Theorem. 

\medskip
The situation in the real case is much more complicated. Even the simplest non-trivial case of {\em unimodal maps}
on the interval, namely smooth maps from a closed interval into itself with a single critical point, whose rigorous study was
started by J.~Milnor and W.~Thurston~\cite{MT88} after the seminal paper on the logistic map by biologist R. May~\cite{May76},
has been fully understood only very recently thanks to fundamental contributions by A.~Avila (see the survey by Lyubich~\cite{Lyu12}
and the references therein).

We notice, first of all, that the characterization of the Julia and Fatou sets is weaker than in the complex case because real
maps are not necessarily open:
\begin{theorem}[see de Melo and van Strien~\cite{dMvS93}, Chapter 5]
  \label{them:dMvS1}
  Let $f:\RPo\to\RPo$ a generic analytical function and denote %by $T_f$ the set of turning points (i.e. maxima and minima) of $f$,
  by $\gamma_i$ the attractive cycles of $f$ and by $W_i$ the set of wandering intervals of $f$, namely those intervals
  whose iterates are all disjoint and that do not converge to a cycle. Then:
  \begin{enumerate}
  \item $F_f$ is backward invariant; 
  \item $J_f$ is forward invariant;
  \item $F_f=\sqcup_i\cF(\gamma_i)\sqcup_j W_j$;
  \item $J_f=\alpha_f(Z_{f})$;
  \item $J_f$ {\em contains} the closure of the set of repelling points of $f$;
  \item $f|_{F_f}$ is almost open in the sense that, if we denote by $U_i$ the connected components of $F_f$, then $f(U_i)\subset U_j$ for some $j$
    if $U_i\cap T_f=\emptyset$ and $f(U_i)\subset \overline{U}_j$ for some $j$ otherwise.
    %  \item $J_f$ is the set of points which neither in a (forward) basin of attraction nor in a wandering interval, namely an interval 
    %    whose forward iterates are all disjoint and that does not converge to a periodic orbit.
  \end{enumerate}
\end{theorem}
%
%\begin{definition}
 % Given a generic smooth map $f:\RPo\to\RPo$ which is not a homeomorphism, we denote by $\Sigma_f$ either the set
%  of all his critical points or, if $f$ has no critical points, the set of its fixed points. The Julia set of $f$ is the closed
%  set $J_f=\alpha_f(\Sigma_f)$ and its Fatou set $F_f$ is the complement of $J_f$ in $\RPo$.
%\end{definition}
%
%\begin{remark}
%  Note that in the real setting in general $F_f$ is only backward invariant, since the points where $f$ is not open
%  might fall on the Julia set, and, correspondingly, $J_f$ is only forward invariant.
%\end{remark}
%
One of the most general results on maps $\RPo\to\RPo$ is the following generalization of Sullivan's
Non Wandering Domain theorem~\cite{MMS92,dMvS93,vSt10}:
\begin{theorem}[Martens, de Melo and van Strien~\cite{MMS92}]
  Let $f:\RPo\to\RPo$ a generic non-invertible $C^2$ map. Then:
  \begin{enumerate}
  \item every connected component of $F_f$ falls, in a finite time, in a periodic component;
  \item there are only finitely many periodic components of $F_f$.
  \end{enumerate}
  Moreover, for almost all $x\in\RPo$, the set $\omega_f(x)$ is of the following three types:
  \begin{enumerate}[label=\roman*.]
  \item a periodic orbit;
  \item a minimal Cantor set;
  \item a finite union of intervals containing a critical point.
  \end{enumerate} 
\end{theorem}
\begin{remark}
  Note that just in 2016 Astorg {\sl et al.} showed~\cite{ABDPR16} that this result is sharp, in the sense that, both in the
  real and complex case, there are polynomials maps $\KPt\to\KPt$, $K=\bR$ or $\bC$, who admit wandering domains. 
\end{remark}
In spite of an extraordinary amount of articles and books devoted to the study of rational maps $\CPo\to\CPo$ in the last forty years,
very few have been dedicated to the general study of arguably the most natural generalization of them, namely rational maps
$\RPt\to\RPt$. Among the few exceptions are the study of Julia sets of dianalytic maps by J. Hawkins et al.~\cite{GH14,Haw15,HR17}
and of the dynamics of a particular family of birational maps by E. Bedford and J. Diller~\cite{BD05,BD06,BD09}. A similar situation
holds in the subcase of Newton maps, that are the subject of this article:
\begin{definition}
  Let $p$ be a polynomial in one variable over real or complex numbers. We call {\em Newton map} associated to $p$ the rational map
  $$N_p(z)=z-p(z)/p'(z).$$
\end{definition}
The Newton's method for finding the root of a function (e.g. see~\cite{Ypm95,Deu11}), of paramount importance in the Numerical Analysis field,
is based on the elementary facts that, for a generic function $p$, the following holds: 1. the set of the roots of $p$ coincides with the set of
(bounded) fixed points of $N_p$; 2. all of these
fixed points are {\em attractive} (in fact, {\em super-attractive}). Hence {\bf iterations of $N_p$ lead naturally to a root of $p$ when the initial
  point is chosen close enough to it} -- see~\cite{Kan49,Deu11} for a very general classical proof of this fact in the context of Banach spaces
and~\cite{HSS01} for a clever algorithm to retrieve {\em all} roots of a complex polynomial based on strong general results of holomorphic dynamics.
\begin{figure}
  \centering
  \begin{tabular}{cc}                                                                                                                                                                  
    \includegraphics[width=6.3cm]{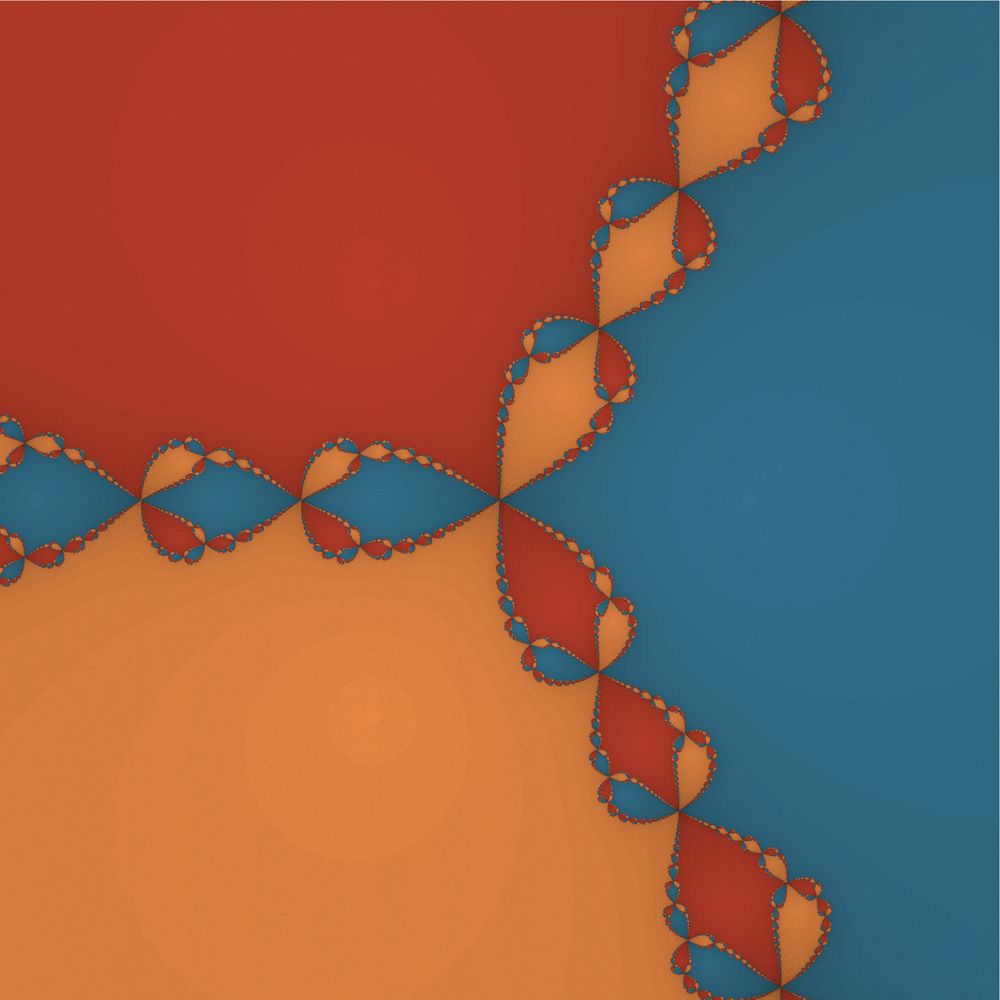}&\includegraphics[width=6.3cm]{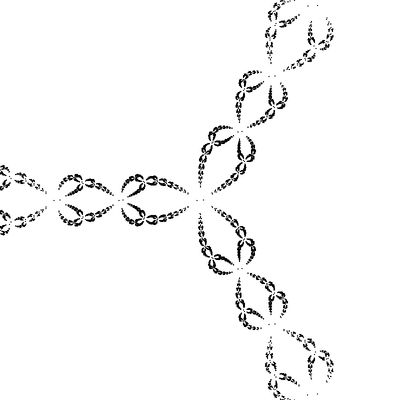}\\
    \includegraphics[width=6.3cm]{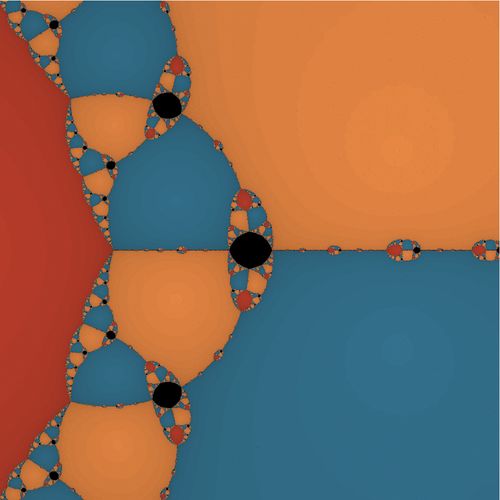}&\includegraphics[width=6.3cm]{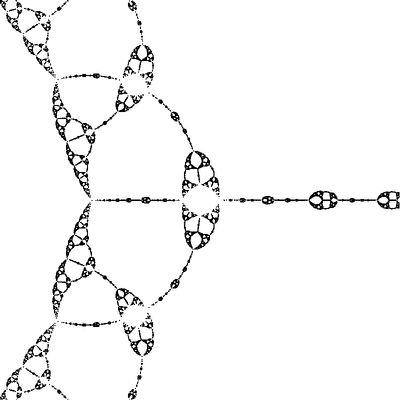}\\
  \end{tabular}
  \caption{%
    \em
    Basins of attraction (left) and Julia set (right) of the Newton maps associated to $p(z)=z^3-1$ in the square $[-2,2]^2$ (first row)
    and $q(z)=z^3-2z+2$ in the square $[-1.5,1.5]^2$ (second row).
    We assigned a color to each root all points belonging to a root's basin has been plotted with that same color.
    The interior of the black islands that appear in case of $q$ does not belong to $J_q$ (see point 3 of Thm~\ref{thm:JF})
    but rather correspond to Fatou components of points that are attracted to non-trivial cycles rather than any of $q$'s roots
    (equivalently, they are attracting basins of some root of $f^k$ for some $k>1$). The right column shows an approximation
    of the $\alpha$-limit set of the point $z_0=5+i$ at the 10th recursion level. 
  }
  \label{fig:cjs}
\end{figure}

Newton maps of complex polynomials are quite special rational functions: for instance, the point at infinity is always a fixed repelling point for them.
The following theorem~\cite{Hea88} provides a full characterization for them:
\begin{theorem}[Head (1988)]
  Every rational complex map $R$ of degree $d$ with $d$ distinct superattracting fixed points is conjugated, via a Mobius
  transformation, to the Newton's map $N_p$ of a polynomial $p$ of degree $d$.
  If $\infty$ is not superattracting for $R$ and $R(\infty)=\infty$, then $R=N_p$.
\end{theorem}
Correspondingly, their Fatou and Julia sets have special properties (e.g. see~\cite{Ruc08}):
\begin{theorem}
  \label{thm:JNp}
  Let $p$ be a polynomial with complex coefficients. Then:
  \begin{enumerate}
%  \item $J_{N_p} = \overline{\{x\in\CPo\,|\; p'\left(N_p(z)\right)=0\hbox{ for some $k\in\bN$}\}}$ 
%  \item The point at infinity is a repellor for $N_p$;
%  \item $J_{N_p} = \bigcap_{n\geq0}\;\overline{\bigcup_{m\geq n}\{(p')^{-m}(0)\}}$ ;
  \item $J_{N_p}$ has empty interior;
  \item $J_{N_p}$ is connected;
  \item all connected component of $F_{N_p}$ are simply connected;
  \item $F_{N_p}$ has no Arnold-Herman rings;
  \item $J_{N_p} = \alpha_{N_p}\left(Z_{N_p}\right)$;
%  \item $J_{N_p} = \overline{\{x\in\CPo\,|\,N_p'(N_p^k(x))=0\hbox{ for some }k\in\bN\}}$ ;
  \item all {\em immediate basins} $B_i$ of the roots of $p$, namely the connected components of $F_{N_p}$ containing those roots,
    are unbounded (i.e. $\infty\in\partial B_i$); 
    \item $\infty$ is a repulsive fixed point for $N_p$.
  \end{enumerate}
\end{theorem}
Notice that it is enough considering polynomials of degree three in order to find cases of Newton maps with
parabolic basins and Siegel discs (e.g. see Sec.~3.2 of~\cite{Ruc08}), although there seem to be no concrete example
available in literature.
\begin{corollary}
  Let $p$ be a generic complex polynomial of degree $n$ with roots $R=\{r_1,\dots,r_n\}$ and such that $N_p$
  has no Siegel domains or attracting $k$-cycles for $k\geq2$ and $J_{N_p}$ has Lebesgue measure zero. Then
  $\alpha_{N_p}(c_i)=J_{N_p}\cup\{c_i\}$, $i=1,\dots,n$, and $\alpha_{N_p}(z)=J_{N_p}$ for any other point.
  Equivalently, $\cB(J_{N_p})=\CPo\setminus R$ and $\cB(J_{N_p}\cup R)=\CPo$. In particular, $N_p$ is
  a \simple\ map.
\end{corollary}
\begin{example}
  This is the case of the Newton maps associated to the polynomials $p(z)=z^n-1$.
  Consider again, for instance, the case of $p(z)=z^3-1$, so that $N_p(z)=\frac{2z^3+1}{3z^2}$ (see Fig.~\ref{fig:cjs}).
  $N_p$ has exactly three attractors, the three roots of unity $u_i$, $i=1,2,3$.
  There cannot be attracting cycles other than these fixed points because, by Theorem~\ref{thm:Fatou}, if there were
  one then a critical point of $N_p$ should converge to it, but for this map $Z_{N_p}$ coincides with the set of zeros of $N_p$.
  Each forward basin $\cF_p(u_i)$ is the disjoint union of countably many simply connected open sets and the boundary
  $\partial\cF_p(u_i)=\overline{\cF_p(u_i)}\setminus\cF_p(u_i)$ of each of them is equal to the Julia set $J_f$.
\begin{figure}
  \centering
  \begin{tabular}{cc}
    \includegraphics[width=6.3cm]{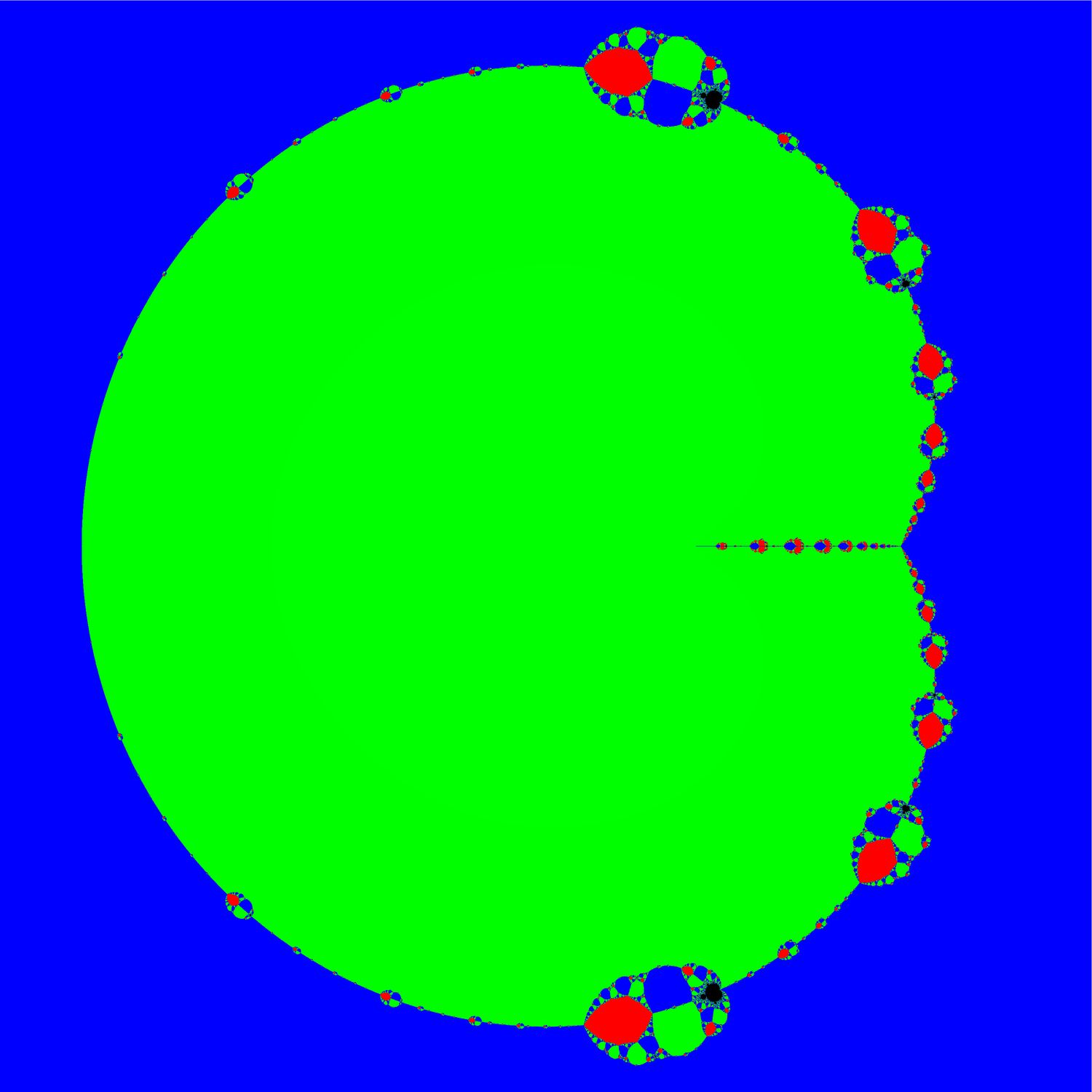}&\includegraphics[width=6.3cm]{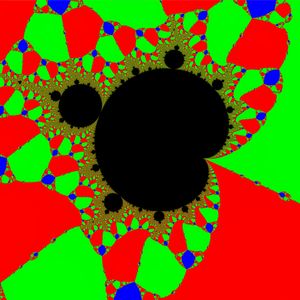}\\
    \includegraphics[width=6.3cm]{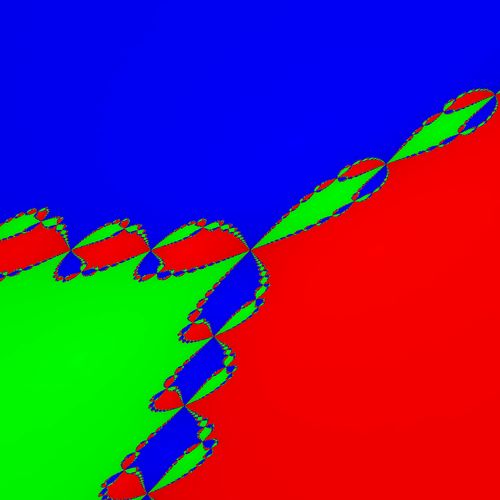}&\includegraphics[width=6.3cm]{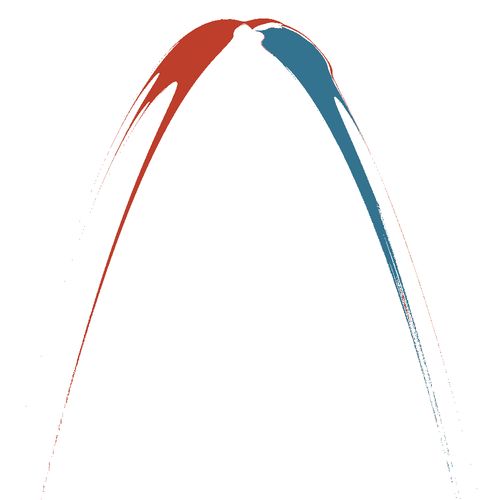}\\
  \end{tabular}
  \caption{%
    \em
    Top: $\omega$-limit of the origin under the Newton map of $f_A(z)=z^3+(A-1)z-A$
    for values of $A$ in the square $[-2.3,1.7]\times[-2,2]$ of the complex plane (left).
    For most values of $A$, the origin converges to one of the three roots but there is a non-zero
    volume set of values (black points) for which the origin converges to a non-trivial attracting cycle.
    The zoom (right) shows that the connected components of this set are the celebrated Mandelbrot fractal.
    Bottom: basins of attraction of the Newton method applied to two intrinsically complex functions.
    On the left, $f(z)=z^3-1$ but the complex structure has been modified so that $i$ is
    represented by the vector $(1,1)$ rather than $(0,1)$. On the right, the function is $\psi^*f$,
    where $\psi(x,y)=(x,y+x^2)$ and $f(z)=z^2-1$.
%    In the first case, since the two complex structures
%    differ jsut by a linear transformation, the Newton map of the original and the transformed map
%    are conjugate, as the picture suggests. The white points in the second case are points attracted by
%    infinity and so in this case the Newton maps of the original and transformed map are not conjugate
%    because infinitey is always repelling for a complex Newton map.
  }
  \label{fig:mps}
\end{figure}

  The Julia set is the only repellor of $N_p$ and $\cB_p(J_f)=\CPo\setminus\{u_1,u_2,u_3\}$.
  In fact, the equation $N_p(w)=z$ has always three solutions (taking into account multiplicity) and this defines three
  meromorphic functions $w_i$ so that $N_p^{-1}(z)=\{w_1(z),w_2(z),w_3(z)\}$. Following Barnsley (see~\cite{Bar88}, Sec. 7.3),
  we can restrict the $w_i$ to the complement of some open neighborhood of the roots of $p$, so that the Iterated
  Function System generated by these restrictions has a unique attractor, which coincides with $J_{N_p}$.
\end{example}
{\bf In this article we are mostly interested in the size of the set of points that do not converge to any root.}
As a consequence of a deep study by Lei Tan of the dynamics of Newton maps coming from cubic polynomials~\cite{Tan97},
we know that there are complex Newton maps whose Julia set has measure larger than zero, although no concrete example appears
in literature so far. On the other side, the following theorem allows to find easily the existence of non-trivial attracting
periodic cycles, whose presence also causes the set of non-converging points to be of non-zero measure:
%The problem of which Julia sets have measure zero is still open and under active reasearch (e.g. see~\cite{Che09}).
%In particular it is not known yet whether the Julia set of Newton maps is always a null set but for us it is enough to look at the
%existence of non-trivial attracting periodic cycles, which is made easy by the following property:
%                                                                                                                                                                                           
%\begin{figure}
%  \centering
%  \includegraphics[width=8cm]{cubicMPS-r600}
%  \caption{%
%    \em
%    Bla bla bla
%  }
%  \label{fig:man}
%\end{figure}
%
\begin{theorem}[Fatou]
  \label{thm:Fatou}
  If a rational map has an attracting periodic cycle, then the orbit of at least one of its critical points will converge to it.
\end{theorem}
When $p$ has degree 1 or 2, the set of non-converging points has trivially measure zero: in the first case, $N_p^n(z)$ converges
to the root for all $z\in\bC$; in the second, $J_{N_p}$ is diffeomorphic to a circle and the Fatou set is the disjoint union of two discs,
each of which is the the immediate basin of one of the two roots. In the first non-trivial case, when $p$ has degree 3, it was found
first numerically by Curry, Garnett and Sullivan~\cite{CGS83} that there are such polynomials whose Newton map
%, even in case of Newton maps coming from degree 3 polynomials, there are cases when there exist indeed
$N_p$ has attracting orbits with period larger than 1 (see Fig.~\ref{fig:mps}). This means that, even for such simple Newton maps, there is an
open set of points (hence with measure larger than zero) that does not converge to any root. A simple example of such
polynomials is $q(z)=z^3-2z+2$ (see Fig.~\ref{fig:cjs}).
%$p(z)=3z^5-10z^3-23z$~\cite{HP88}.

\medskip
Of course the restriction of complex polynomials with real coefficients to the real line provides examples of dynamics of
{\em real} Newton maps on $\RPo$, so the example above shows that such behavior also takes place in the real case.
%$3^{rd}$ degree {\em real} polynomials.
It is, therefore, non-trivial and particularly interesting the following result found by B\'ela Barna~\cite{Bar53},
way before the explosion of work on complex holomorphic dynamics:
\begin{theorem}[Barna, 1953]
  Let $p$ be a generic real polynomial of degree $n\geq 4$ without complex roots and denote its roots by $r_1,\dots,r_n$. Then:
  \begin{enumerate}
  \item $F_{N_p}=\cup_{i=1}^n\cF(c_i)$;
  \item $F_{N_p}$ has full Lebesgue measure;
  \item $N_p$ has no attractive $k$-cycles with $k\geq2$;
  \item $N_p$ has repelling $k$-cycles of any order $k\geq2$;
  \item $J_{N_p}$ is equal, modulo a countable set, to a Cantor set $\cE_{N_p}$ of Lebesgue measure zero.
  \end{enumerate}
%  In particular, $N_p$ is a \simple\ map.
  %  Let $p$ be a generic real polynomial of degree at least 4 with Newton map $N_p$ % having $d$ real roots
 % and denote by $\Lambda_{N_p}$ the set of points of $\RPo$ that do not converge to any root under $N_p$.
 % Then $\Lambda_{N_p}$ is the union of the countable set $S_{N_p}=\cup N_p^{-m}(\{p'(x)=0\})$ with a Cantor set
 % $\cE_{N_p}$ of zero Lebesgue measure so that, in particular, $\Lambda_{N_p}=J_{N_p}$.
 % Moreover, $N_p$ has periodic points of all periods.
\end{theorem}
\begin{remark}
  In fact, it is more generally true that the, for any complex polynomial $p$ with all its roots $\{r_1,\dots,r_n\}$
  simple and real, $F_{N_p}=\cup_{i=1}^n\cF(c_i)$ has full Lebesgue measure (\cite{Lyu86}, Thm.~1.27).
  Even more generally, $J_{N_p}$ is a set of Lebesgue measure zero if all critical points of $N_p$ converge to
  attracting, repelling or neutral rational cycles (\cite{Lyu86}, Thm.~1.26). E.g. this last theorem covers all polynomials
  $p_n(z)=z^n-1$, $n=2,3,\dots$,showing that all the $p_n$ are \simple.
\end{remark}
%  almost every point in $\CPo$ coverges, under $N_p$, to one of the roots of $p$
%  Julia set $J_{N_p}\subset\CPo$ of any complex polynomial $p$ with all roots simple and real
%  $J_{N_p}\subset
In our knowledge, the only result in literature about $\alpha$-limits of maps on $\RPo$ is that $J_f=\alpha_f(Z_f)$ (see Thm.~\ref{them:dMvS1}).
It is reasonable, though, to believe that this property extends to almost all points of $\RPo$, leading to the following:
\begin{conjecture}
  Every $p$ satisfying the hypothesis of Barna's theorem is a \simple\ map.
\end{conjecture}
About thirty years later, Barna's work was revisited independently in the same year by Saari and Urenko~\cite{SU84},
Wong~\cite{Won84} and Hurley and Martin~\cite{HM84} leading, in particular, to the following important
results:
\begin{theorem}[Wong, 1984]
  A sufficient condition for Barna's theorem to hold is that the polynomial $p$ has no complex root and at least 4 distinct real roots,
  possibly repeated.
\end{theorem}
\begin{theorem}[Saari and Urenko, 1984]
  Let $p$ be a generic polynomial of degree $n\geq3$, $A_p$ the collection of all bounded intervals in $\bR\setminus Z_{p}$
  and $\cA_p$ the set of all sequences of elements of $A_p$. Then the restriction of $N_p$ to the Cantor set $\cE_{N_p}$ is semi-conjugate to the
  one-sided shift map $S$ on $\cA_f$, namely there is a surjective homomorphism $h_p:\cE_{N_p}\to A_p$ such that $T\circ h_p=h_p\circ N_p$.
\end{theorem}
\begin{theorem}[Hurley and Martin, 1984]
  Let $p$ be a generic polynomial of degree $n\geq3$. Then $N_p$ has at least $(n-2)^k$ $k$-cycles for each $k\geq1$ and the
  topological entropy of $N_p$ is at least $\log(n-2)$.
%  Let $p$ be a polynomial, $N_p$ its Newton map and $B_p$ the set of the two unbounded intervals in $\bR\setminus\{f'(x)=0\}$.
%  Unless both intervals in $B_p$ contains some fixed point of $N_p$, the map $N_p$ has positive topological entropy.
\end{theorem}
\begin{remark}
  The theorem by Saari and Urenko actually holds for the much larger class of ``polynomial-like'' functions
  and similarly happens for the Hurley and Martin theorem (see~\cite{SU84} and~\cite{HM84} for details).
\end{remark}
%
%In~\cite{HM84,Won84} Barna's Theorem was revisited and generalized to polynomials with multiple
%roots but {\bf
Despite the depth and interest of these results for Newton maps on the real line, {\bf no attempts to multidimensional generalizations of Barna's theorem
  appear in literature. In the next section we present numerical evidence showing that a similar statement might hold in higher dimension,
  or at least on the real plane.}

%
%\section{Numerical investigations and conjectures on Newton maps on the real plane}
\section{Newton maps on $\Rt$}
%\section{General properties and conjectures on Newton maps on $\Rt$}
%
The Newton method extends naturally to much more general settings than the real and complex lines,
from finite-dimensional linear spaces~\cite{Hub15} to Banach spaces~\cite{Kan49,Deu11}
to Riemannian manifolds~\cite{Fer02}. Moreover, since Newton's map $f\mapsto N_f$ leaves invariant
the subset of {\em complex} maps of $\bR^{2n}$ into itself, in each setting one can consider separately
the real and the complex case.

In the present article we are only interested in the finite-dimensional real case:
\begin{definition}
  Let $f:\K^n\to\K^n$, $\K=\bR$ or $\bC$, be a polynomial map.
  We call {\em Newton map} associated to $f$ the rational map $N_f:\KPn\to\KPn$ defined by
  $$
  N_f(x)=x-D_xf^{-1}\left(f(x)\right).
  $$
\end{definition}
Moreover, we will limit our discussion to the case $n=2$. Notice that {\bf very little, compared to the 1-dimensional case,
is available in literature about Newton's method in $\bR^2$ or $\bC^2$.} The complex case has been recently
investigated in a few articles by Hubbard and Papadopol~\cite{HSS01} and by Hubbard's pupil Roeder~\cite{Roe05,Roe07},
where they classify and study of the case of {\em quadratic polynomials}. As expected, technical difficulties
are quite more challenging than in dimension one. The real case was considered, to the author's knowledge, only by
Peitgen, Prufer and Schmitt~\cite{PR86,PPS88,PPS89} about thirty years ago, mostly from the point of view
of identifying the best definition of {\em Julia set} in the real multidimensional context, and about twenty years
ago by Miller and Yorke~\cite{MY00}, that studied the size of attracting basins.

{\bf In this work we are rather interested to a somehow transversal point of view, namely for which real polynomial maps
  of the plane into itself the corresponding Newton maps are \simple.}
%  as similar as possible to the one of complex maps on the line.}

To begin with, from a real point of view, complex differentiable maps $f:\bC\to\bC$ are just real differentiable maps
$f_\bR:\Rt\to\Rt$ whose Jacobian $Df_\bR$ commutes with the (imaginary unit) matrix
$$
J=\begin{pmatrix}
\phantom{-}0&1\\
-1&0\\
\end{pmatrix}
$$
or, equivalently, for which $f_\bR$ we have that 
$$
Df_\bR=\begin{pmatrix}\phantom{-}\alpha&\beta\\ -\beta&\alpha\\ \end{pmatrix}
$$
for some $\alpha,\beta\in C^0(\Rt)$. Hence all results of complex holomorphic dynamics apply
to such functions and, more generally, we can state the following proposition:
\begin{proposition}
  Let $f:\Rt\to\Rt$ a smooth map that is {\em intrinsically complex}, namely complex with respect
  to some almost complex structure of the plane.
%  Equivalently, let $f=g_\bR\circ\psi$ for some holomorphic function $g$ and some real diffeomorphism $\psi$ or $\bR^2$ into itself.
%  Then either $J_g$ is empty or all topological properties that hold for the iterations of holomorphic maps also hold for iterations of $f$;
%  in particular, if $g$ is rational, the Julia set of $f$ is never empty and is equal to the boundary of the basins of its attracting cycles.
  Then exactly one of the following holds:
  \begin{enumerate}
  \item $f$ enjoys all topological properties that hold in holomorphic dynamics 
    %its iterates converge to a constant function or $f=g_\bR\circ\psi$
    %  for some holomorphic function $g$ all topological properties that hold for the iterations of holomorphic maps also hold for iterations of $f$;
    (in particular, its Julia set is never empty and it is equal to the boundary of any of the basins of its attracting cycles).
  \item $f$ is conjugate via a diffeomorphism to a rotation about a fixed point.
  \item $\{f^n\}$ converges to a constant function in the Withney topology.
  \end{enumerate}
  Note that, in the last two cases, $J_f$ is empty.
\end{proposition}
\begin{proof}
  This is an immediate consequence of a few important results: the integrability of all almost complex structures in dimension 2
  (since, by purely dimensional reasons, the Nijenhuis tensor is identically zero), the Uniformization Theorem and the Denjoy-Wolff Theorem
  (see~\cite{Mil06} for more details about the last two). By the first, every almost complex structure gives rise to a complex structure,
  so that every almost complex map, namely every map $f$ that preserves the almost complex structure, is a complex map with respect
  to the corresponding complex structure.
  By the second, this complex structure must be diffeomorphic to one of the following two inequivalent ones: either the standard one
  on the unit open disc or the standard one on the whole plane. In the first case, $f$ is smoothly conjugated to a standard complex
  function and therefore all topological results of holomorphic dynamics also apply to it. In the second case, $f$ is conjugated to a
  holomorphic map of the unit disc in itself and, by the Denjoy-Wolff Theorem, this means that it is either conjugated to a
  (hyperbolic) rotation or its iterates converge, uniformly on compact sets, to a constant function.
%  Since every almost complex bidimensional manifold
%  is integrable and, by the Uniformization Theorem, there are exactly two complex structures on $\bR^2$ modulo diffeomorphisms,
%  then every intrinsically complex function $f$ is conjugated via a diffeomorphism to complex either with thre
\end{proof}
{\bf Since the Newton map is not natural with respect to diffeomorphisms, namely $N_{\psi^*f}\neq \psi^* N_f$,
  where $\psi^* f=\psi^{-1}f\psi$ and similarly for $N_f$, the proposition above does not apply, in general,
  to Newton maps of intrinsically complex maps.}
\begin{example}
  \label{ex:sa}
  In case of the (standard) complex map $f(z)=z^2-1$, namely $N_f(z)=(z^2+1)/(2z)$, it was shown already in late eighteen hundreds independently by
  E. Schr\"oder~\cite{Sch70,Sch71} and A. Cayley~\cite{Cay79a,Cay79b} that $J_{N_f}\subset\CPo$ is a circle and
  there are exactly two (connected) basins of attraction, corresponding to the two roots $\pm1$. The (real) Julia set in $\RPt$
  of the corresponding map $f(x,y)=(x^2-y^2-1,2xy)$ is, in turn, the wedge of two circles, since the point at infinity of $\CPo$ expands
  to a whole circle in $\RPt$ and this circle is left invariant by $N_{f}$. Written in homogeneous coordinates, the real version of this Newton
  map reads
  $$
  N_f([x:y:z]) = [x (x^2 + y^2+z^2): y (x^2 + y^2-z^2):2 z (x^2 + y^2)].
  $$
  In the projective chart $y=1$, the circle at infinity becomes the $x$ axis and the restriction of $N_f$ to it is the identity map.
  A direct calculation shows that the Jacobian of $N_f$ at each point of the $x$ axis is diagonal with eigenvalue 2 in the $z$
  direction, corresponding to the fact that the point at infinity of a complex map is always repelling.

  Under the diffeomorphism $\psi(x,y)=(x,y+x^2)$, $f$ transforms into
  $$
  \psi^*f(x,y) = (x^2 - (y + x^2)^2 - 1,2 x (y + x^2) - (x^2 - (y + x^2)^2 - 1)^2)\,
  $$
  which is a complex map with respect to the almost complex structure
  $$
  J(x,y) = -2x\,\partial_x\otimes dx -\partial_x\otimes dy + (1+4x^2)\,\partial_y\otimes dx +2x\,\partial_y\otimes dy.
  $$
  Clearly $f_\psi=\psi^*f$ and $f$ enjoy the same dynamical and topological properties but, on the contrary, $N_f$ and $N_{f_\psi}$ turn out to
  be very different. For instance, the map $N_{f_\psi}$ has only one fixed point at infinity, the point $[0:1:0]$, and this point is actually
  superattracting (see Fig~\ref{fig:mps}), something that could not happen to a complex Newton map.
\end{example}
As already pointed out by Yorke {\sl et al.} in~\cite{MGOY85}, complex maps (as well as intrinsically complex ones) are very special
among real maps and it is not to be expected that their asymptotic behavior is shared by general real maps. Obvious algebraic
differences are that standard generic complex polynomial maps of {\em complex degree} $n$ are {\sl balanced}, in the sense
that both of its components are real polynomials of the same degree $n$, and have $n$ distinct roots while, seen
as real polynomial maps\footnote{By {\em real degree} of a polynomial map on the plane we mean the product of the degrees of its components},
they have degree $n^2$ and a generic real map of this type
%real polynomial map\footnote{By {\em real degree} of a polynomial map on the plane we mean the product of the degrees of its components}
%has degree $n^2$
%can have up to $n^2$ distinct roots. Dynamical differences are not less critical. For instance, as pointed out by Peitgen et al. in~\cite{PPS89},
can have up to $n^2$ distinct roots. Moreover, Newton maps of general polynomials maps on the plane have {\em points of indeterminacy},
namely points where they are necessarily undefined; this last difference, though, is not essential because, by a standard result in algebraic geometry,
all such points can be eliminated by a finite sequence of blow-ups. 
From the dynamical point of view, as pointed out by Peitgen et al. in~\cite{PPS89},
{\bf unlike in the complex case, points at infinity can be attracting for real Newton maps on the plane}, even in case
of Newton maps of {\em intrinsically} complex ones (e.g. see the example above). In particular, the circle at infinity is not
necessarily contained inside the Julia set of a real Newton map.
%Moreover, correspondingly with the fact
%that real polynomials can have derivative always different from zero, {\bf the Julia set of a real Newton map can be empty}
%(this is true also in one dimension, e.g. see ???).
%Furthermore, as noticed by Hubbard et al. in~\cite{HSS01} in the complex case,
%ensions not all Newton maps $\K^2\to\K^2$ extend to maps $\KPt\to\KPt$ because of the appearance of {\em indeterminacy points}
%(see ???).

In fact, it is expected that the dynamics of general {\em real} maps on the plane %, even in case of just polynomial maps of degree 2,
be more complicated and diverse than the one of holomorphic maps, even just because of the presence of {\em hyperbolic} fixed points.
{\bf It is therefore natural to ask whether, besides intrinsically complex ones, there are other classes of real
  polynomial (and, more generally, rational) maps on the plane whose behavior is comparable, if not simpler, to the one
  of holomorphic maps on the complex line.}
%Similarly to the one-dimensional real case, we define the real Julia and Fatou sets as follows:
%Note that, in general, in the real case Theorem~\ref{thm:JF} does not hold and it is not clear yet what is the appropriate definition in general for
%the real version of Julia set. Following Peitgen~\cite{PPS89}, we define the real Julia and Fatou sets as follows:
%
%\begin{definition}
 % Let $f:\RPt\to\RPt$ be a generic rational map and $\Sigma_f$ the set of its critical points. Then $J_f=\alpha_f(\Sigma_f)$
 % and $F_f=\RPt\setminus J_f$.
  %with roots $\alpha_1,\dots,\alpha_n$.
%  Then the Fatou and Julia sets of $f$ are, respectively, $F_f=\cup_{i=1}^n\cF(\alpha_i)$ and $J_f=\bR^2\setminus F_f$.
%\end{definition}
%%
%\begin{remark}
%  Note that, due to point (5) of Theorem~\ref{thm:JNp}, this definition coincides with the one on the complex line for intrinsically complex maps.
%\end{remark}
%
%Note that {\bf these Julia sets are, in general, much larger than their complex counterpart} because, besides the points
%close to which the iterates of $f$ are not normal, these ones also include the open sets where $f$ is normal but its iterates do not converge
%such as wandering domains, Siegel disks and Hermann rings.
    
Based on the numerical results we present below and motivated by Barna's Theorem for the 1-dimensional case,
%we studied numerically 
we pose the following conjectures.
\begin{definition}
  We say that a point $p$ of the Julia set $J_F$ of a rational map $F:\Rt\to\Rt$ is {\em regular} if there is a neighborhood
  $U$ of $p$ such that $J_F\cap U$ is a connected 1-dimensional submanifold and $U$ contains points from two different basins.  
\end{definition}
\begin{conjecture}
  Let $f:\Rt\to\Rt$ be a generic polynomial map of degree $n\geq3$.
  Then there is some non-empty open subset $A\subset f(\RPt)$ such that $\alpha_{N_f}(x)$ is equal to the set of non-regular points
  of the boundary of $J_{N_f}$ for all $x\in A$.
\end{conjecture}
%
%\begin{remark}
%  In the complex case all points of Julia sets are irregular.
%  In the real case there is no hope that the $\alpha$-limit of almost every point of $f(\RPt)$ be equal to $J_f$, as the case of the
%  map $f(x,y)=(y-x^2,x-y^2)$, discussed in Section~\ref{sec:numerics}, clearly shows.
%\end{remark}
%
%\begin{remark}
%  If Conjecture 1 is true, then the same is true in one dimension. We are not aware of any result about $\alpha$-limits of
%  generic points for maps $f:\RPo\to\RPo$.
%\end{remark}
%
\begin{conjecture}
  Let $f:\Rt\to\Rt$ be a polynomial map of degree $n\geq3$ with $n$ distinct real roots $\{c_i\}$. Then:
%  and let $N_f:\RPt\to\RPt$ be its Newton map. Then:
  %
  \begin{enumerate}
  \item $J_{N_f}$ is countable union of wedge sums of countable number of circles and of Cantor sets of circles of measure zero;
  \item $F_{N_f}$ has no wandering domains;
  \item the union of the basins of attraction $\cF_{N_f}(c_i)$ has full Lebesgue measure;
%    union of a a piecewise smooth Cantor set of Lebesgue measure zero (so, in particular, $N_f$ has no wandering domains);
%  \item the complement of $\cB(J_{N_f})$ has Lebesgue measure zero;
%  \item $\cB(J_{N_f})=\RPt\setminus\{r_i\}$;
  \item every neighborhood of any point of $J_{N_f}$ contains points from at least two distinct basins of attractions;
  \item unlike the holomorphic case:
    \begin{enumerate}
    \item basins of attractions are not necessarily simply connected;
    \item immediate basins of attraction are not necessarily unbounded;
    \item $J_{N_f}$ can have interior points without being equal to the whole $\RPt$.
    \end{enumerate}
%  \item  $J_{N_f}$ is connected -- equivalently, every connected component of $F_{N_f}$ is simply connected;
%  \item the {\em immediate basins} of the roots of $f$, unlike in the complex case, are not necessarily unbounded.
  \end{enumerate}
  In particular, $N_f$ is weakly \simple.
\end{conjecture}
\section{Numerical Results}
In Figures~\ref{fig:q1} to~\ref{fig:misc} we show some numerical results, supporting the conjectures above, on the
$\omega$- and $\alpha$-limits of points under iterations of real Newton maps associated to polynomials maps
of various degrees in two variables.

Every row (with the exception of Fig.~\ref{fig:misc} and the middle row in Fig.~\ref{fig:q1}) shows, next to each other, the basins of attraction
of a Newton map (left), with a different color associated to each attractor, and the  $\alpha-$limit of a point under that
Newton map (right) in black and white. The striking resemblance between each $\alpha-$limit and the geometry
of the corresponding basins of attraction, regardless of the number of real solutions, supports Conjecture~2.

Figures~\ref{fig:q1}(top),\ref{fig:c1}(top and middle) and~\ref{fig:c2}(top and middle) show the basins of attraction of Newton
maps corresponding to real polynomial maps with maximal number of real solutions (respectively 4, 6, 6, 9 and 9).
%$$
%f(x,y) = ()
%\,,\;
%g(x,y)
%\,,\;
%h(x,y)
%\,,\;
%p(x,y)
%\,,\;
%q(x,y)
%$$
These pictures strongly suggest the following facts:
\begin{enumerate}
\item regular points of the Julia set cannot be reached via $\alpha-$limits (Fig.~\ref{fig:c1}(top));
\item every neighborhood of every point of the Julia set contains points from at least two basins (Fig.~\ref{fig:c3}(bottom)
  suggests that this is not the case when the number of real roots is not maximal);
\item boundaries between basins of attractions are smooth, except at countably many nodal points;
\item basins of attraction are not necessarily simply connected (Fig.~\ref{fig:c1}(middle));
\item immediate basins are not necessarily unbounded (Fig.~\ref{fig:c2}(top)).
\end{enumerate}
%
%There are similaritiesand differences with the intrinsically holomorphic case: like in that case, basins of attraction seem simply connected
%and the $\alpha$-limit of some open set of points is equal to the non-regular points of the Julia set (in the holomorphic
%case all points of the Julia set are non-regular); unlike that case, boundaries between basins of attractions look smooth,
%except at some countable set of nodal points, and immediate basins of attractions can be bounded (e.g. the blue and
%purple in~\ref{fig:c1}(top) and the cyan one in~\ref{fig:c2}).

Figures~\ref{fig:q1}(bottom),\ref{fig:c1}(bottom),\ref{fig:c2}(bottom) and~\ref{fig:c3}(all) show the basins of attraction
of Newton maps corresponding to polynomials with fewer roots than maximal.
Although we know from the intrinsically holomorphic case that, for some polynomials, basins of attraction can satisfy
the same properties of those with maximal number of roots (e.g. see~\ref{fig:c2}(bottom)), numerics strongly suggest
that this is not always the case.

The fourth order polynomial map $g$ corresponding to Fig.~\ref{fig:q1}(bottom) has only two real roots but its Newton map
$N_g$ has three basins of attraction:
the green and the red ones correspond to its two roots while the cyan one corresponds to some more complicated attractor.
This attractor, possibly a Cantor invariant set, lies on an invariant line of the Newton map, namely the intersection
with the real plane $Im(z)=Im(w)=0$ of the complex line in $\Bbb C^2$ passing through the two complex solutions
of $g$. We call such lines {\em ghost lines}. Points in the cyan basin show high sensitivity to initial conditions
and so belong to $J_{N_g}$ rather than $F_{N_g}$.
This suggests that, in the real case, the Julia set can have a non-empty interior without being necessarily the whole $\RPt$
(as it happens, instead, for intrinsically holomorphic maps).

In Fig.\ref{fig:misc} (top, left) we show the main elements of the dynamics of $N_g$, namely: roots of $g$ (red), indeterminacy
point (blue), invariant line through the two roots (light blue), invariant ghost line (light brown), the hyperbola of points sent
to infinity (light blue) with its 1st (red) and 2nd (brown) counterimages and the first 500 points (green) of the orbit of a random
point in the cyan basin of attraction.

The $\alpha$-limit of points in some non-empty open set seems to be equal to the boundary of the Julia set,
{\bf suggesting a more suitable split of $\RPt$ in case of real maps: $\RPt=A_{N_f}\sqcup R_{N_f}$, where
  $A_{N_f}$ is the union of the Fatou set with {\em all} basins of attraction and $R_{N_f}$ its complement.}
Numerics suggest that, with this definition, the set $R_{N_f}$, as for $J_{N_f}$ in the intrinsically holomorphic case,
has no interior points, possibly even when $f$ has no real roots at all, and, for some non-empty open set,
the $\alpha$-limits of points is equal to its non-regular points.
Observe that, since in the holomorphic case Newton maps are always non-chaotic on their basins of attraction,
this split coincides with the split in the Fatou and Julia sets in the intrinsically holomorphic case.

In Figures~\ref{fig:c1}(bottom) and~\ref{fig:c3}(top) the basins of attraction are intertwined in such a way to
suggest a {\em Cantor set of circles} structure with non-zero measure for the corresponding Julia set.
Following numerically the evolution of the $\omega$-limits in one-parametric families close to the bifurcation
point where a couple of real roots disappear, we observed that usually the basins of attraction of the disappeared
roots get replaced by the basin of attraction of a Cantor set lying in some neighborhood of an invariant ghost line
%the intersection of the complex line passing through the two complex roots in the plane and the real plane $Im(z)=Im(w)=0$
(see the middle and bottom rows of Fig.~\ref{fig:misc}) and then the size of this basin usually decreases in favor
of the basins of the remaining real roots.

In Fig.\ref{fig:c3}(middle) we show the two basins of the Newton map of a map $g$ of degree 9 with a single real root.
In this case the basin of attraction of the only root (in red) is bounded and connected (but not simply connected) while the
other one is the basin of attraction of a Cantor set (in cyan) in some neighborhood of the union of the four
ghost lines corresponding to the four pairs of mutually 
%intersections of the complex lines passing through the four pairs of mutually
conjugated pairs of complex solutions.
%and the real plane $Im(z)=Im(w)=0$.
The $\alpha$-limit seems to be equal to the difference between the Julia set and the basin of attraction of the Cantor set.

Finally, in Fig.\ref{fig:c3}(bottom) we show the basin of attraction (in red) of the Newton map of a map $h$
of degree 9 with a single root and its Julia set. Even in this case we can find points whose $\alpha$-limit is
equal to the Julia set. 

\begin{figure}
  \centering
  \begin{tabular}{cc}                                                                                                                                                                  
    \includegraphics[width=5.6cm]{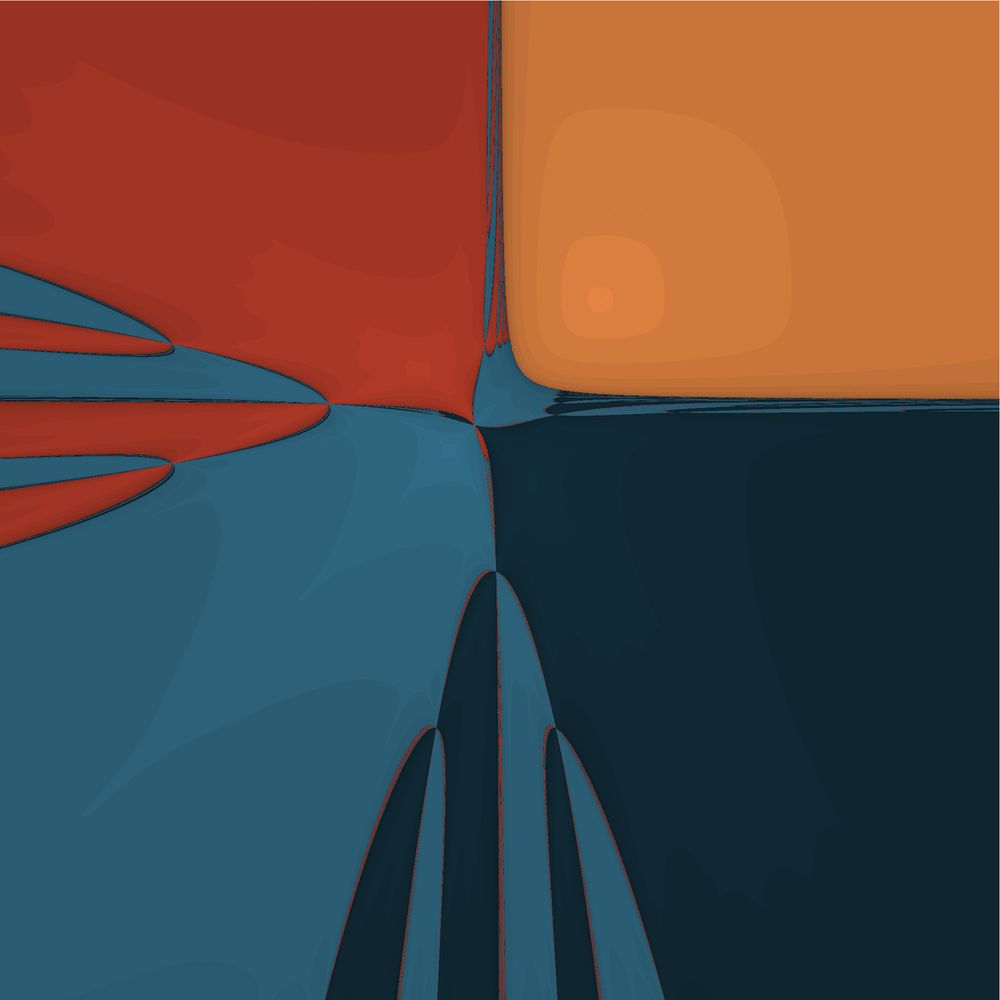}&\includegraphics[width=5.6cm]{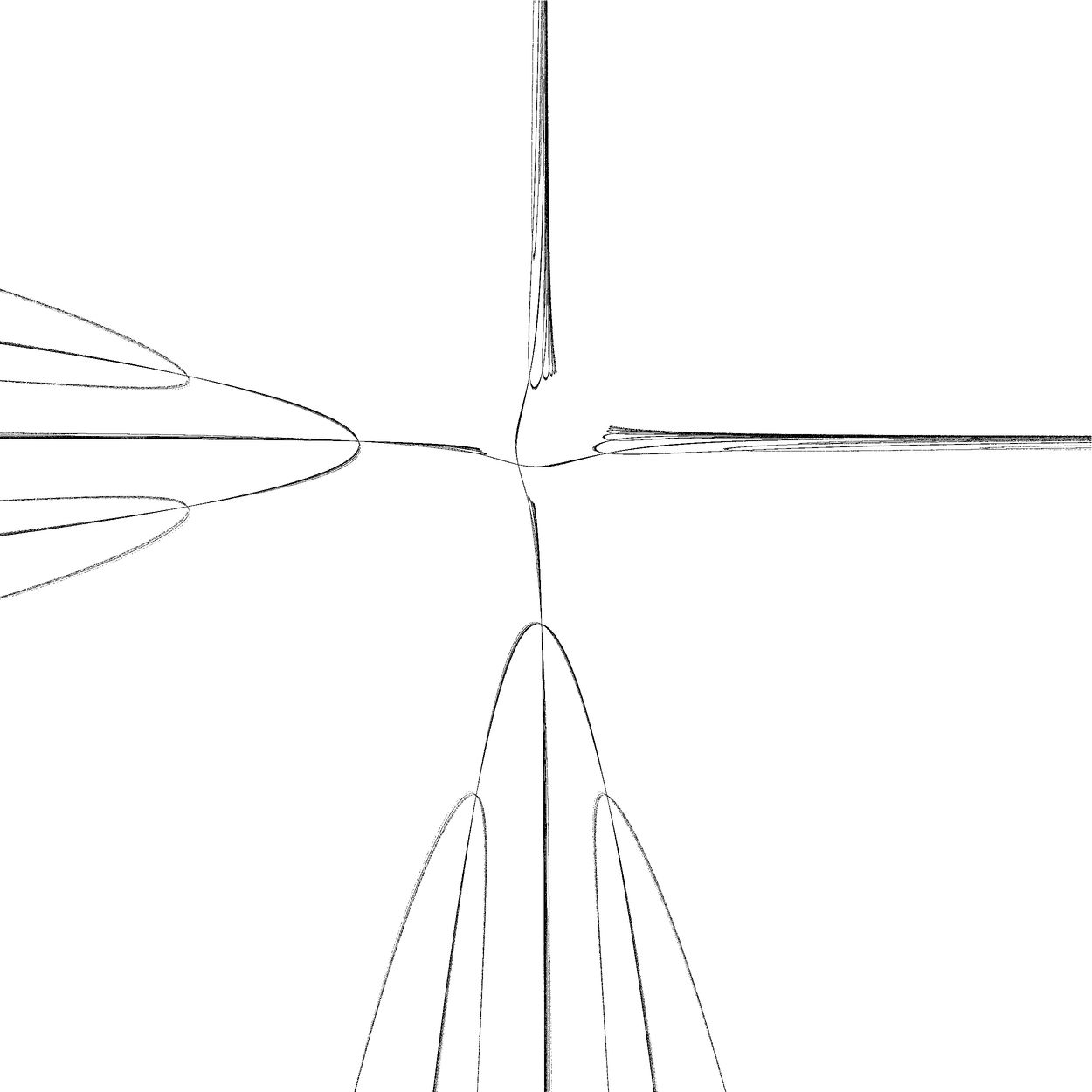}\\
    \includegraphics[width=5.6cm]{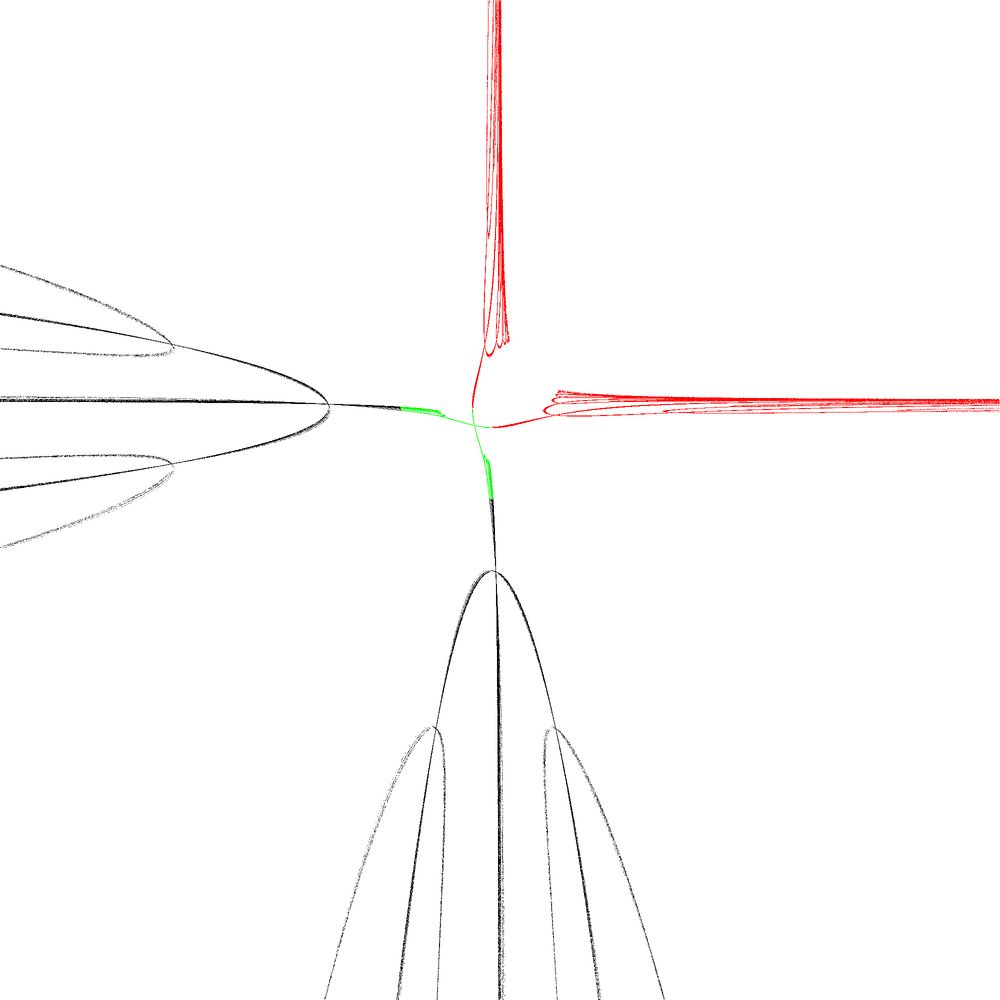}&\includegraphics[width=5.6cm]{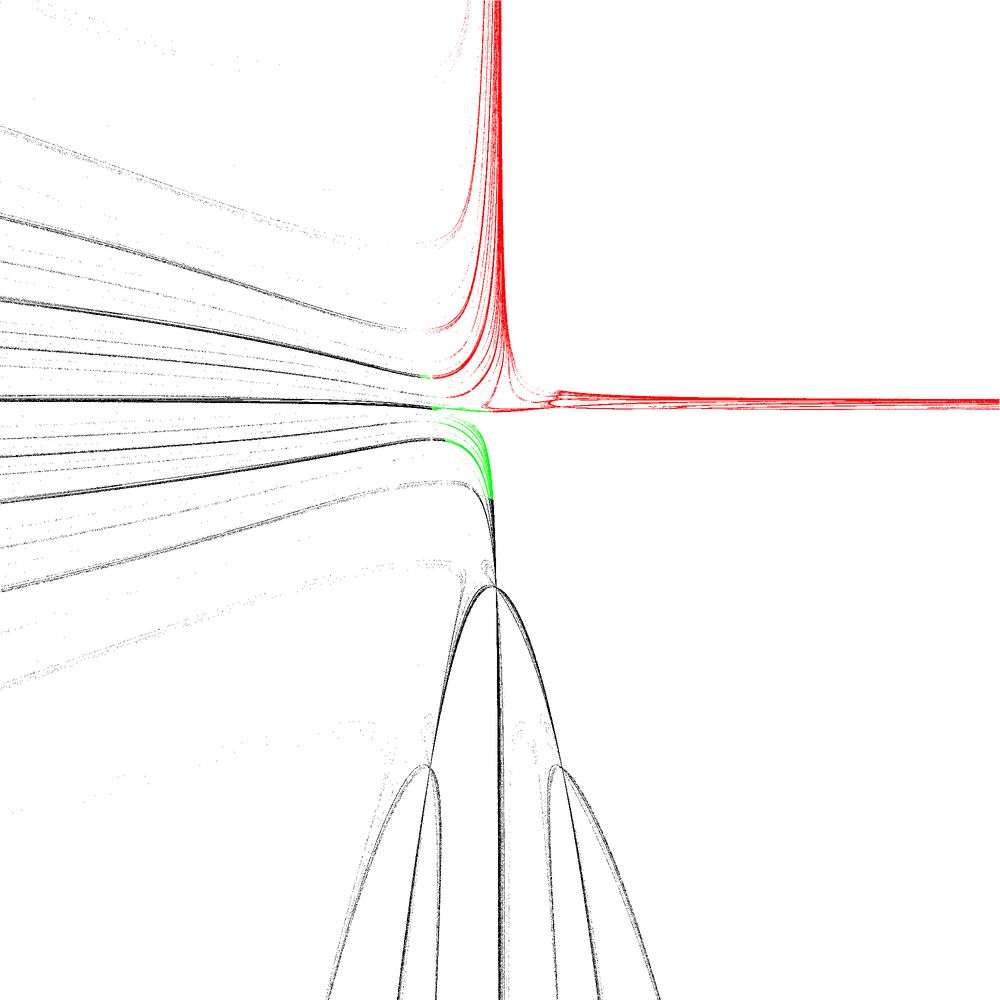}\\
    \includegraphics[width=5.6cm]{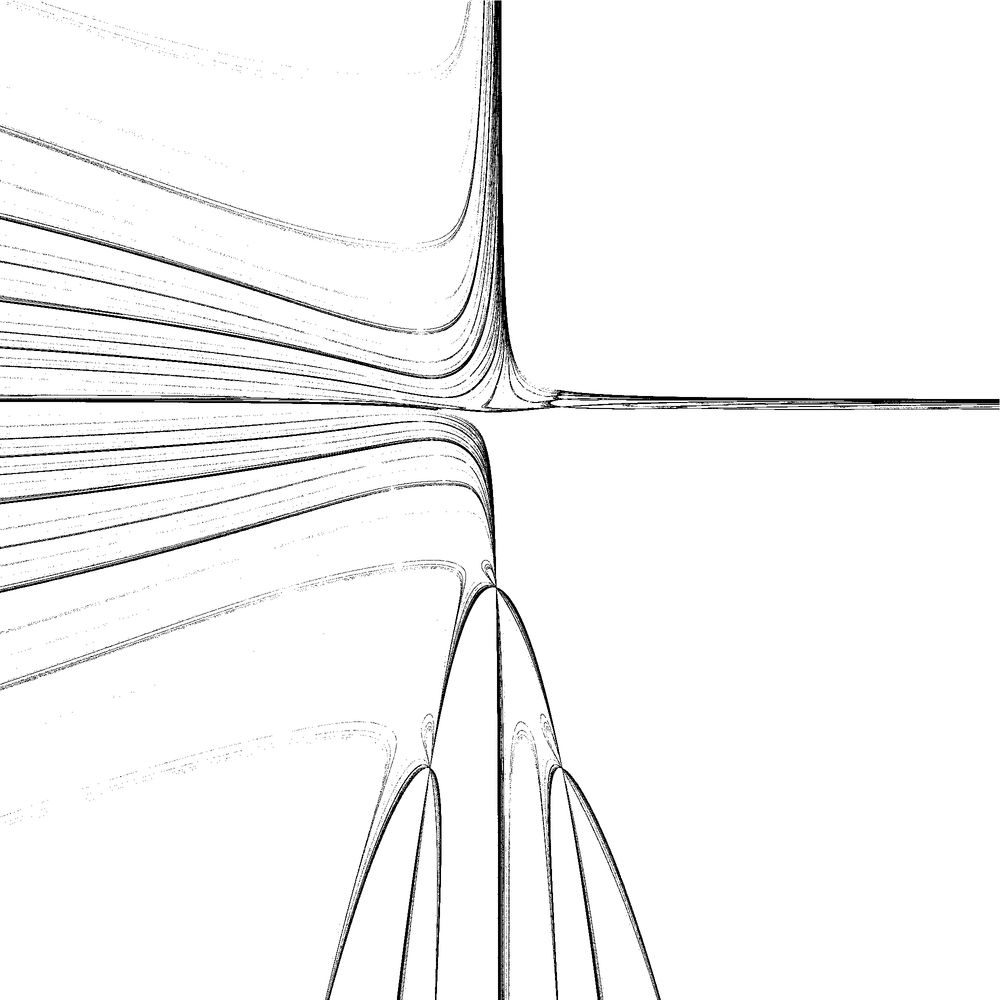}&\includegraphics[width=5.6cm]{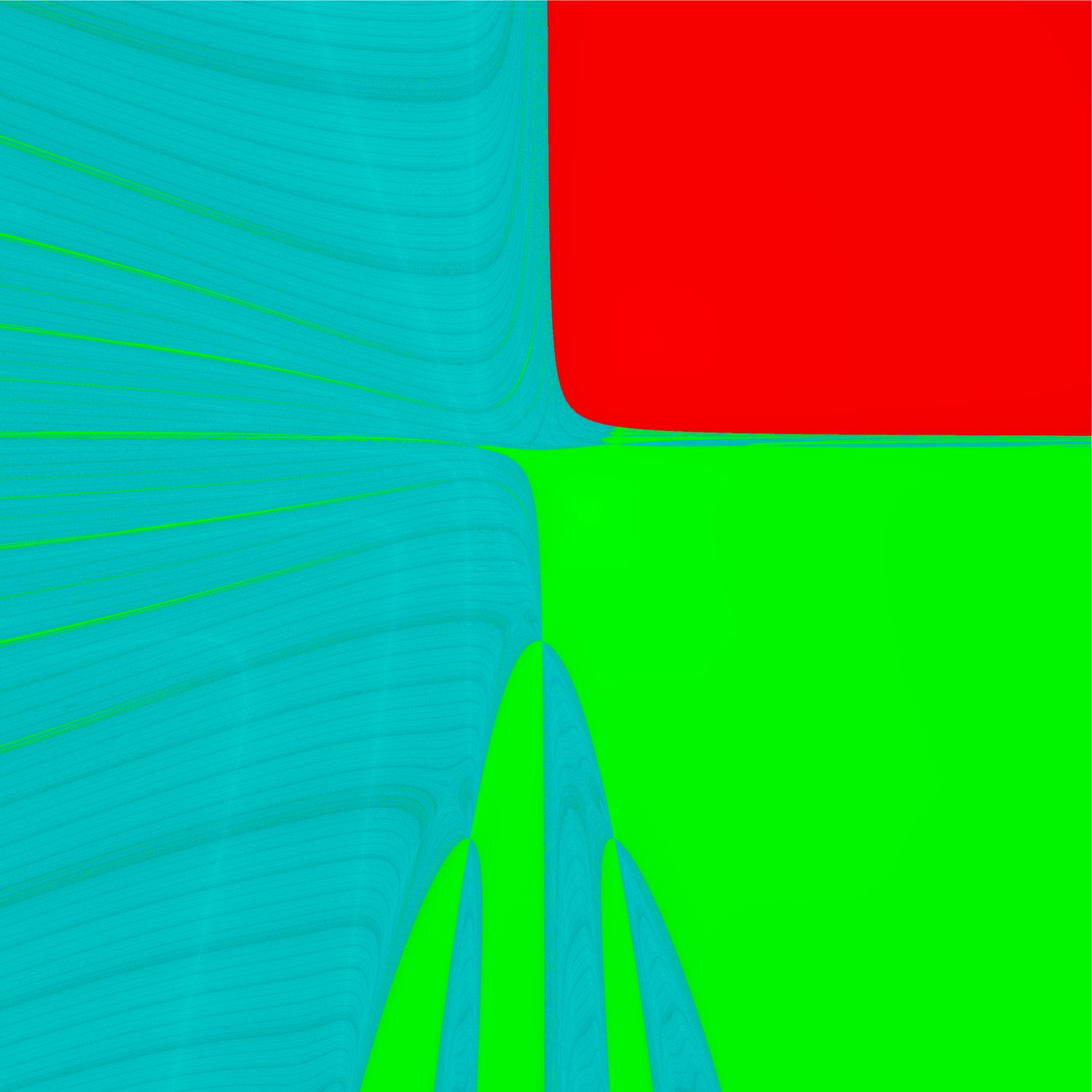}\\
  \end{tabular}
  \caption{%
    \em\small
    $\omega$-limits (in color, each color corresponding to a different basin of attraction) and $\alpha$-limits (in black and white)
    for the Newton maps of the polynomial maps $f(x,y)=(y-x^2,x+2-(y-2)^2)$ (first row) and $g(x,y)=f(x,y)-(0,1)$ (last row).
%    In the middle row
%    we show the unique invariant compact set of the Iterated Function System~\cite{Bar88} obtained by restricting three branches
%    of $N_f^{-1}$ (left) and $N_g^{-1}$ (right) to some suitable invariant set (in black). The invariant set is found numerically by
%    taking a random backward path (see~\cite{HT03} for more details on this technique). The red set is the image of the invariant
%    set under the fourth branch of $N_f^{-1}$ and $N_g^{-1}$. The green points are points on the random path that have no counterimage
%    (in that case another point is chose at random among the real counterimages).
  }
  \label{fig:q1}
\end{figure}
\begin{figure}
  \centering
  \begin{tabular}{cc}                                                                                                                                                                  
    \includegraphics[width=5.6cm]{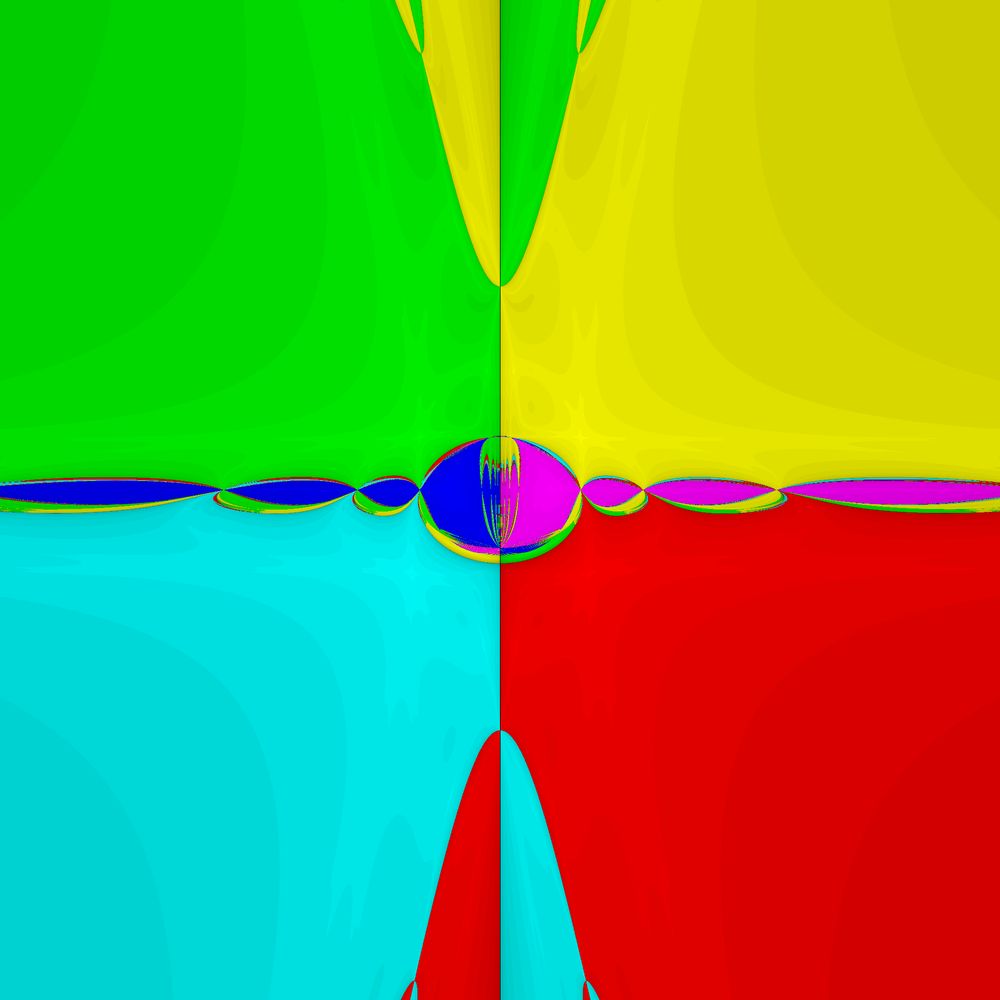}&\includegraphics[width=5.6cm]{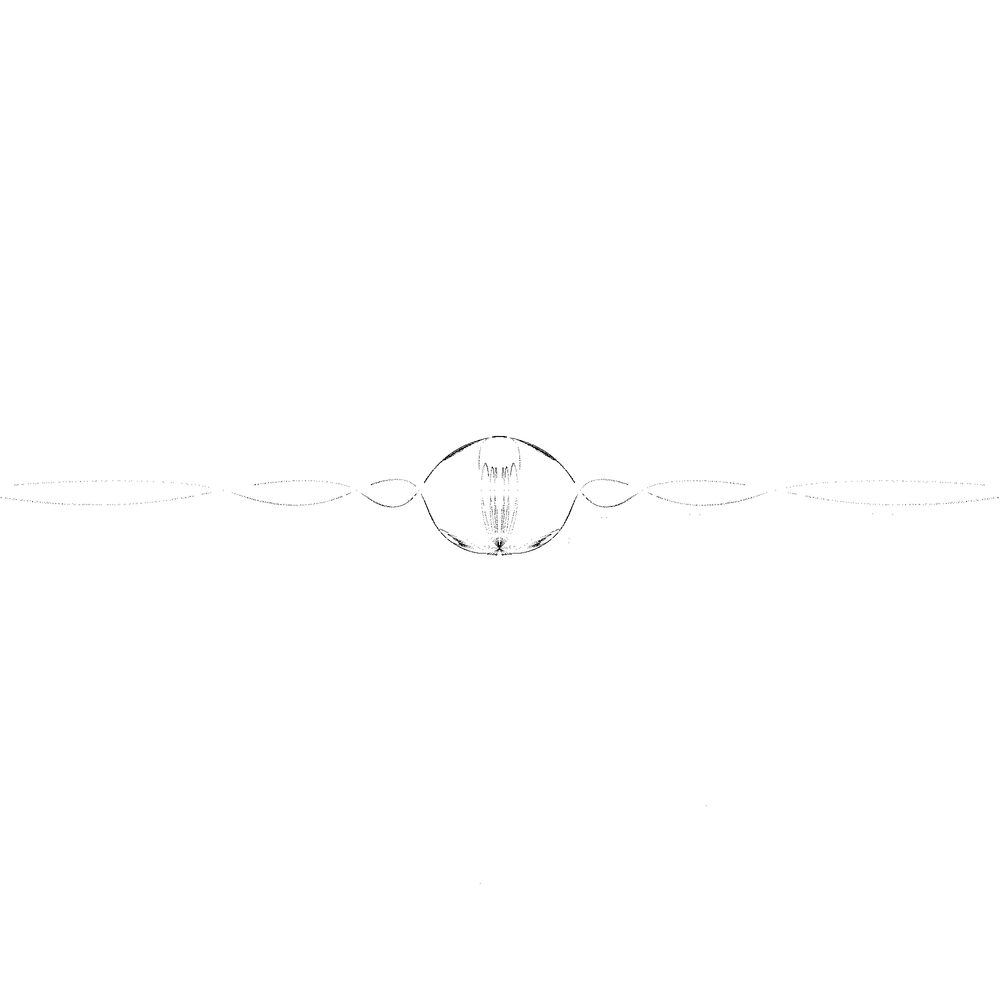}\\
    \includegraphics[width=5.6cm]{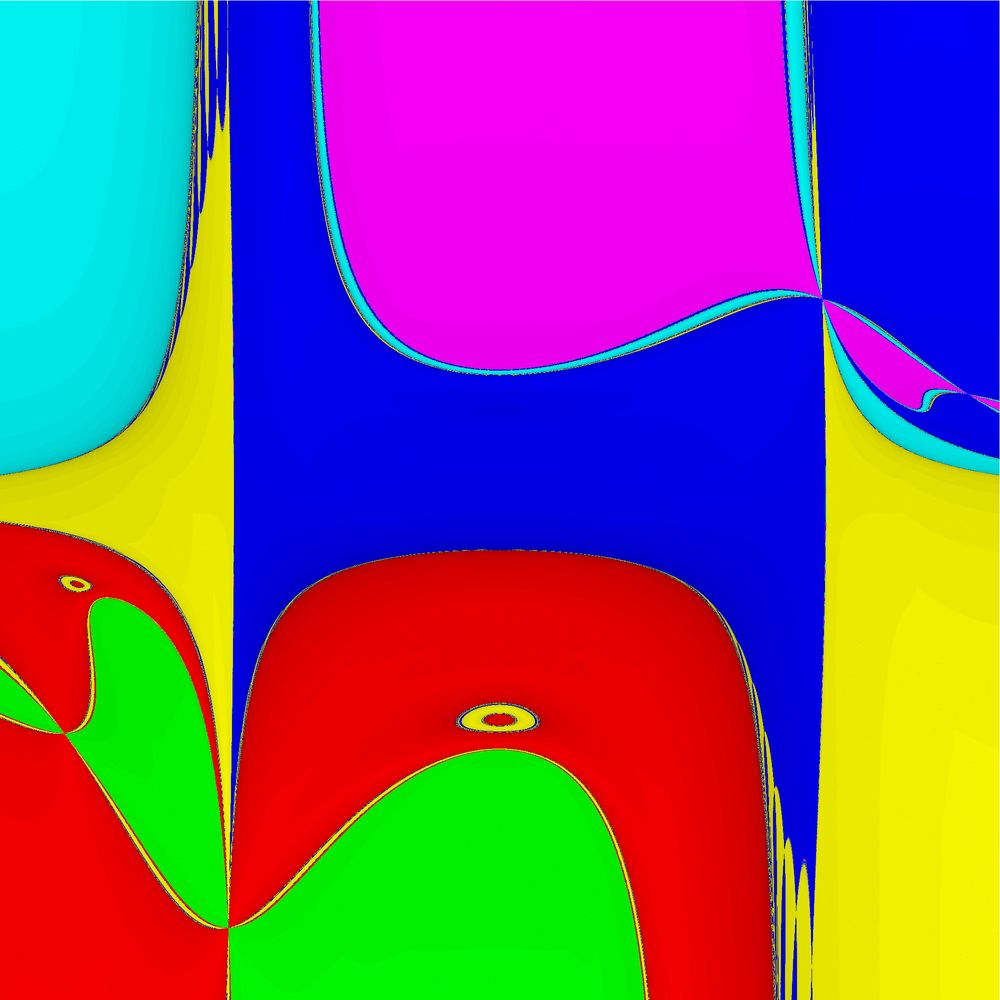}&\includegraphics[width=5.6cm]{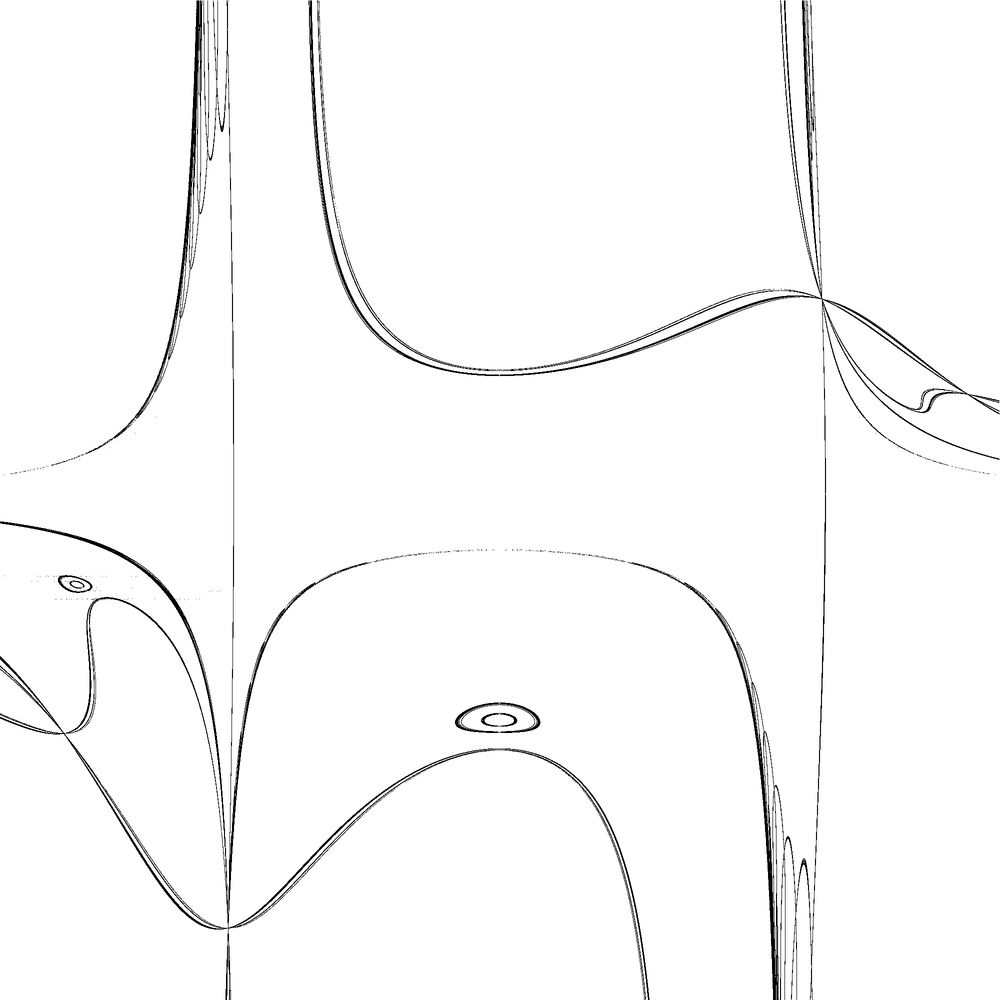}\\
    \includegraphics[width=5.6cm]{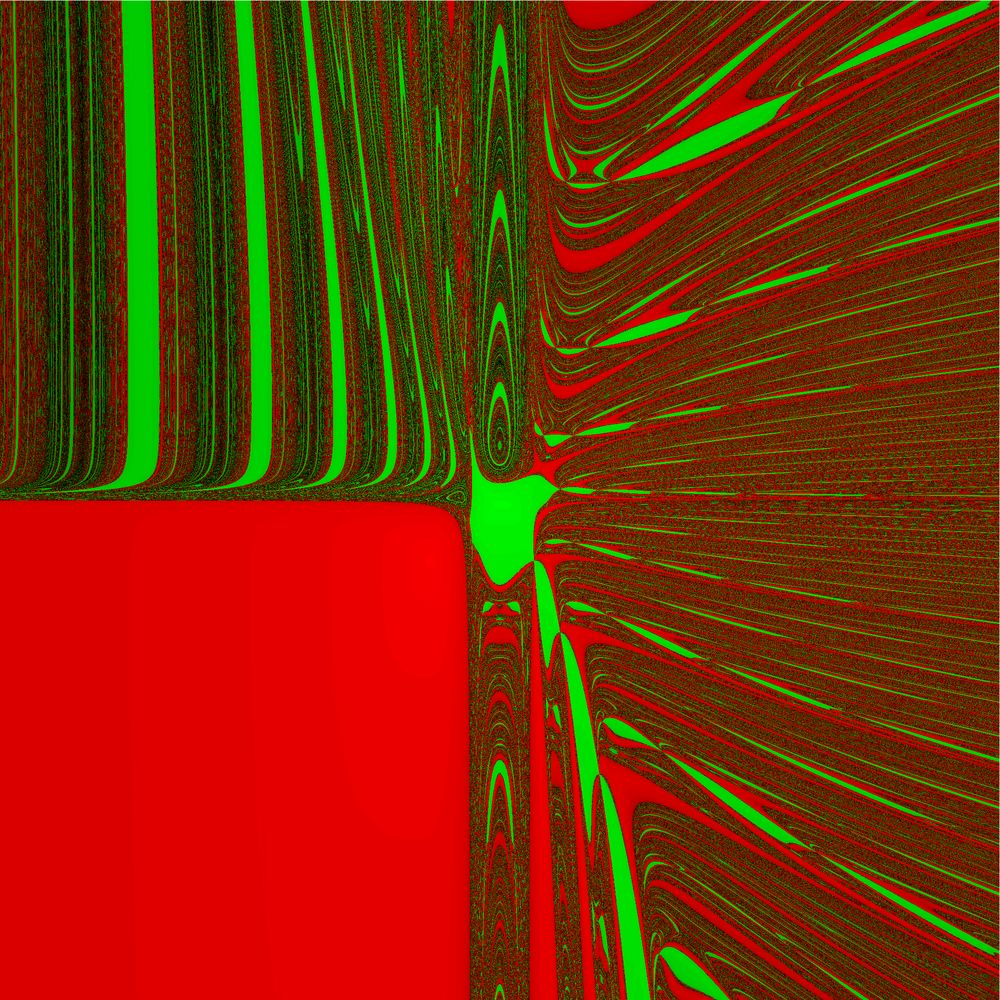}&\includegraphics[width=5.6cm]{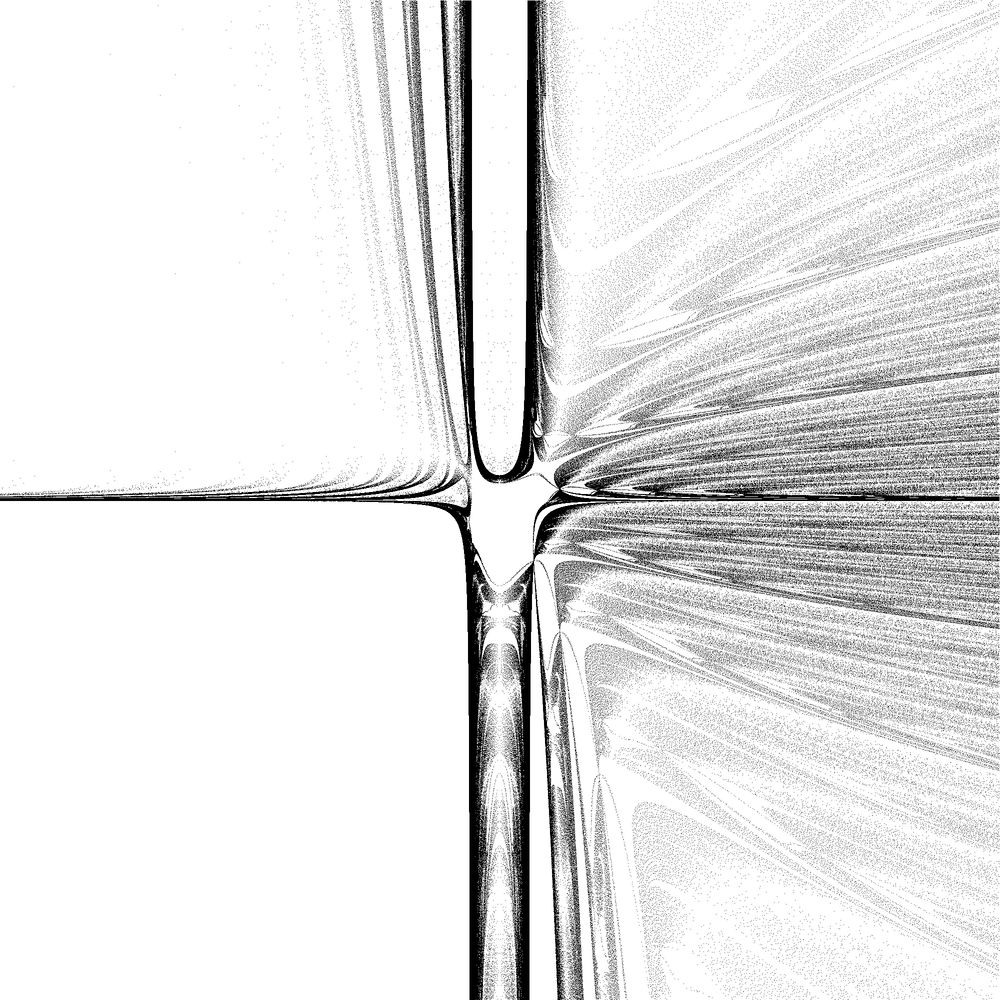}\\
  \end{tabular}
  \caption{%
    \em\small
    $\omega$-limits (in color, each color corresponding to a different basin of attraction) and $\alpha$-limits (in black and white)
    for the Newton maps of  the polynomial maps (from top to bottom) $f(x,y)=(6-9x^2+24y-9x^2y+9y^2+y^3,x^2+y^2-6)$,
    $g(x,y)=(5x(x^2-1)+y,y^2+x-2)$ and $h(x,y)=(5x(x^2-1)-5y,10(y^2+x)-1)$.
  }
  \label{fig:c1}
  \end{figure}
\begin{figure}
  \centering
  \begin{tabular}{cc}                                                                                                                                                                  
    \includegraphics[width=5.6cm]{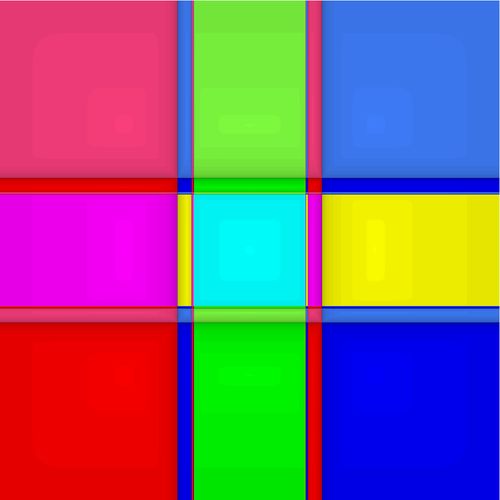}&\includegraphics[width=5.6cm]{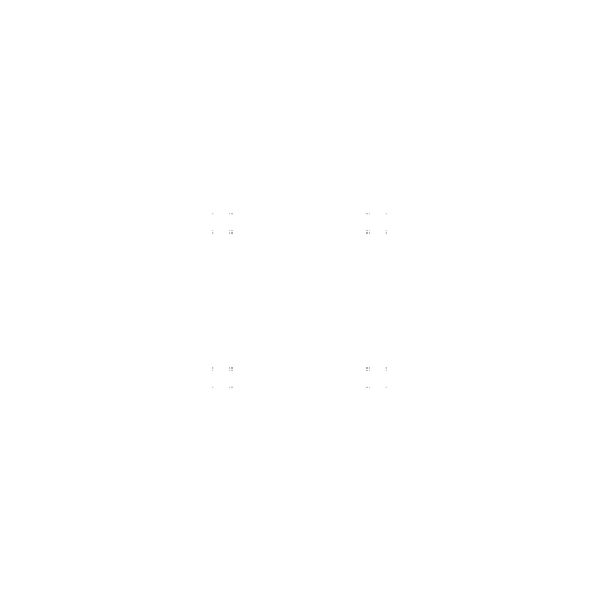}\\
    \includegraphics[width=5.6cm]{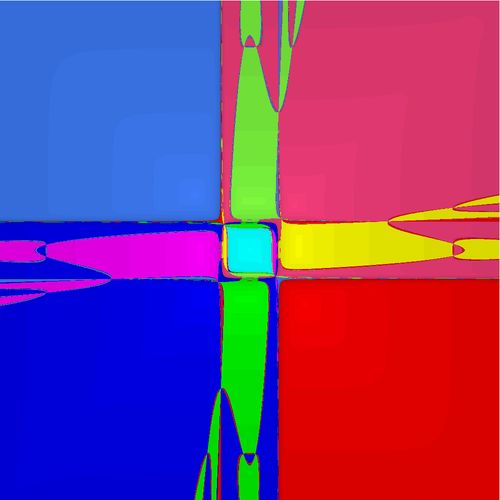}&\includegraphics[width=5.6cm]{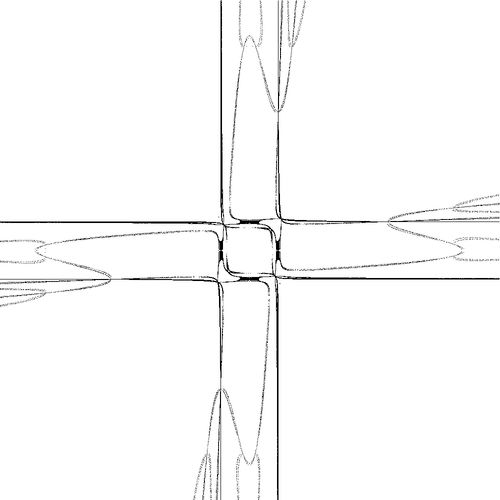}\\
    \includegraphics[width=5.6cm]{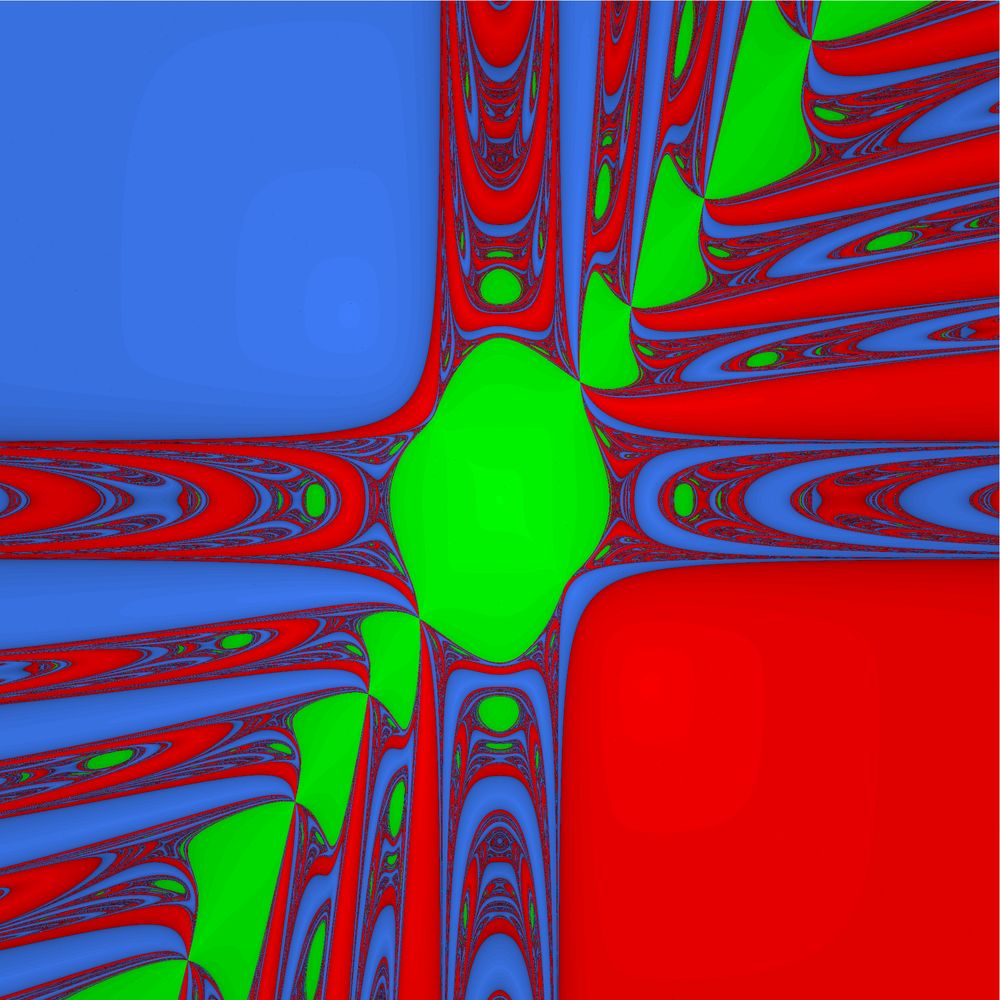}&\includegraphics[width=5.6cm]{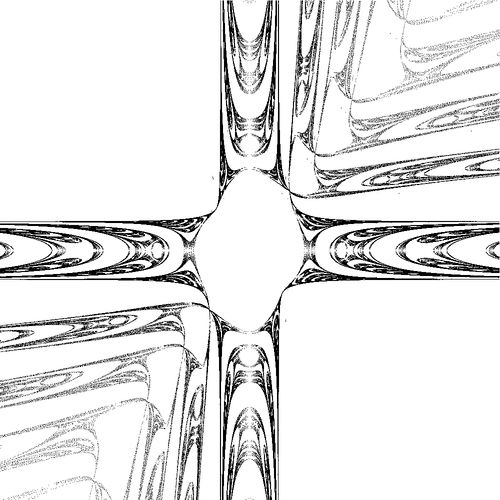}\\
  \end{tabular}
  \caption{%
    \em\small
    $\omega$-limits (in color, each color corresponding to a different basin of attraction) and $\alpha$-limits (in black and white)
    for the Newton maps of  the polynomial maps (from top to bottom) $f(x,y)=(x(x^2-1),y(y^2-1))$,
    $g(x,y)=(20x(x^2-1)+y,20y(y^2-1)+x)$ and $h(x,y)=(x(x^2-1)+y,y(y^2-1)+3x)$.
  }
  \label{fig:c2}
  \end{figure}
\begin{figure}
  \centering
  \begin{tabular}{cc}                                                                                                                                                                  
    \includegraphics[width=5.6cm]{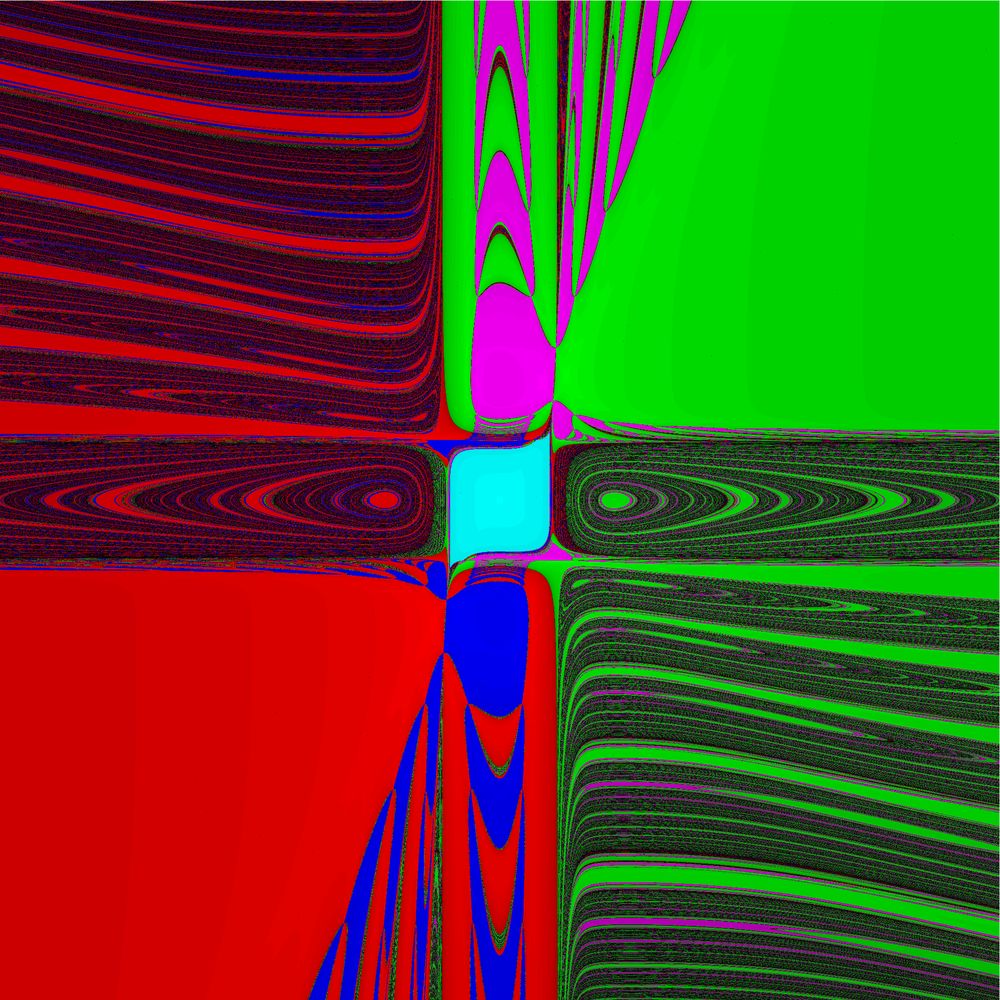}&\includegraphics[width=5.6cm]{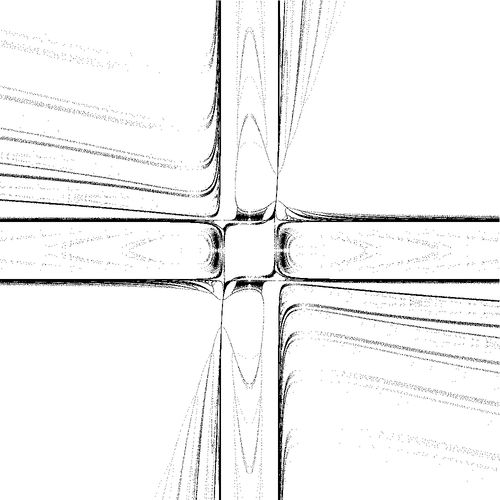}\\
    \includegraphics[width=5.6cm]{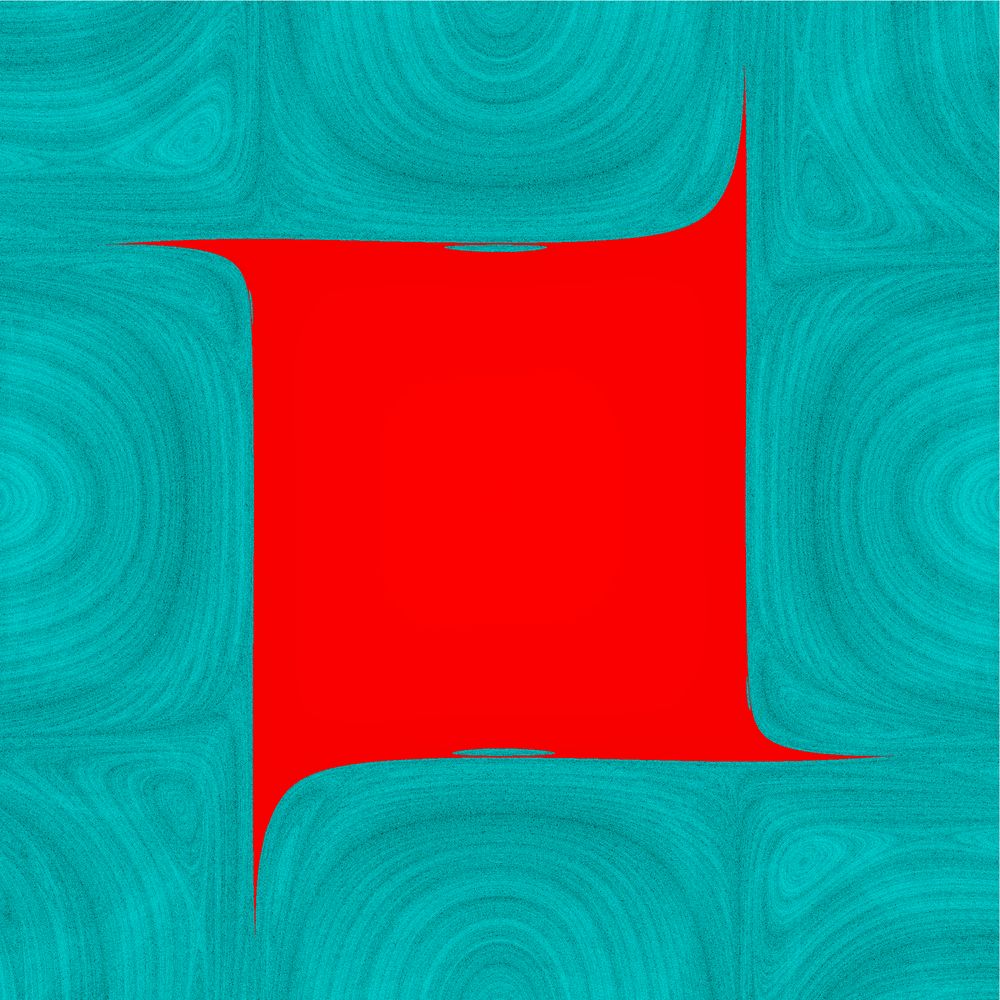}&\includegraphics[width=5.6cm]{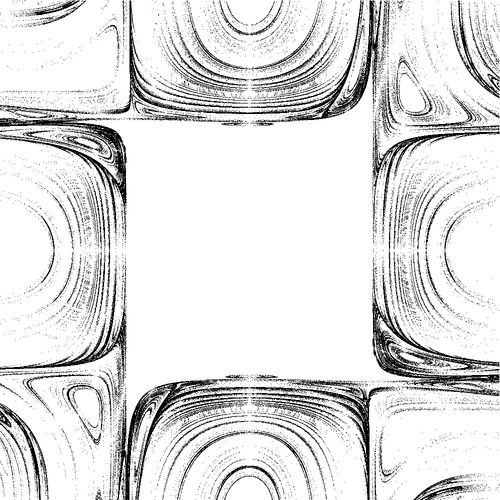}\\
    \includegraphics[width=5.6cm]{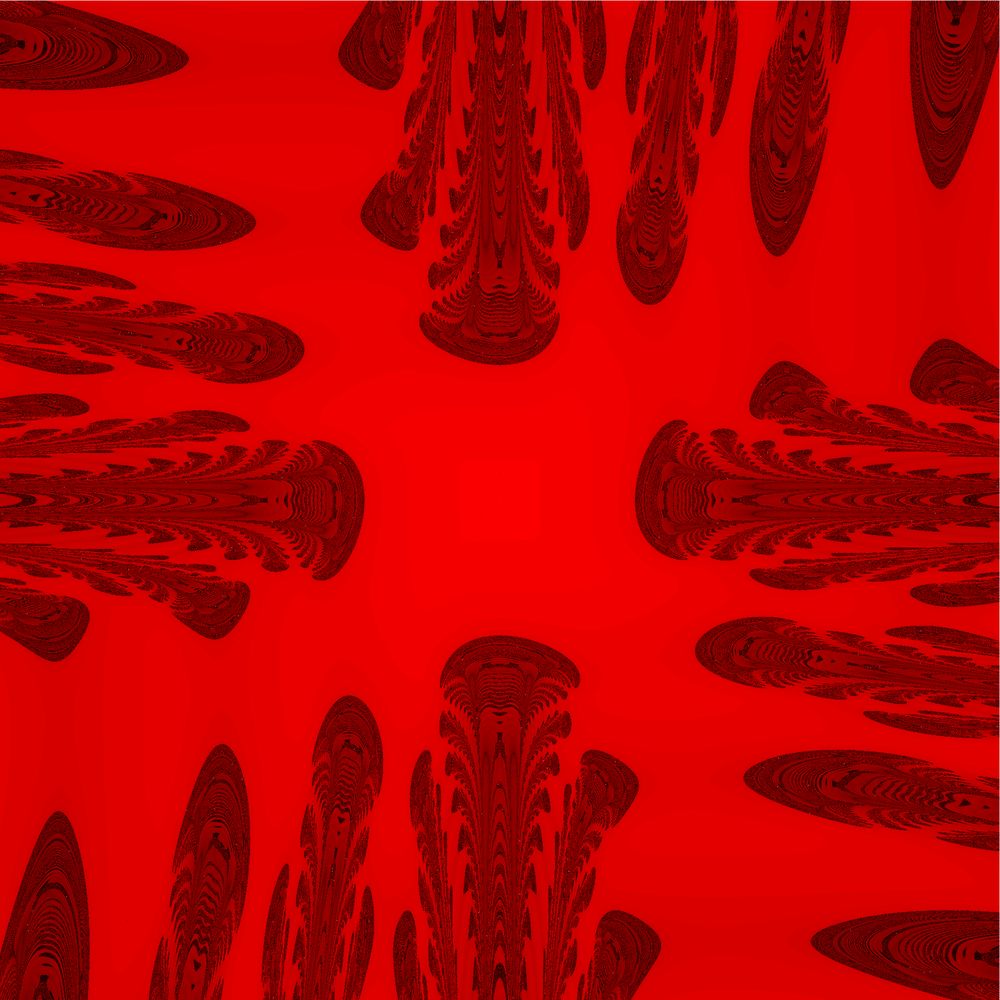}&\includegraphics[width=5.6cm]{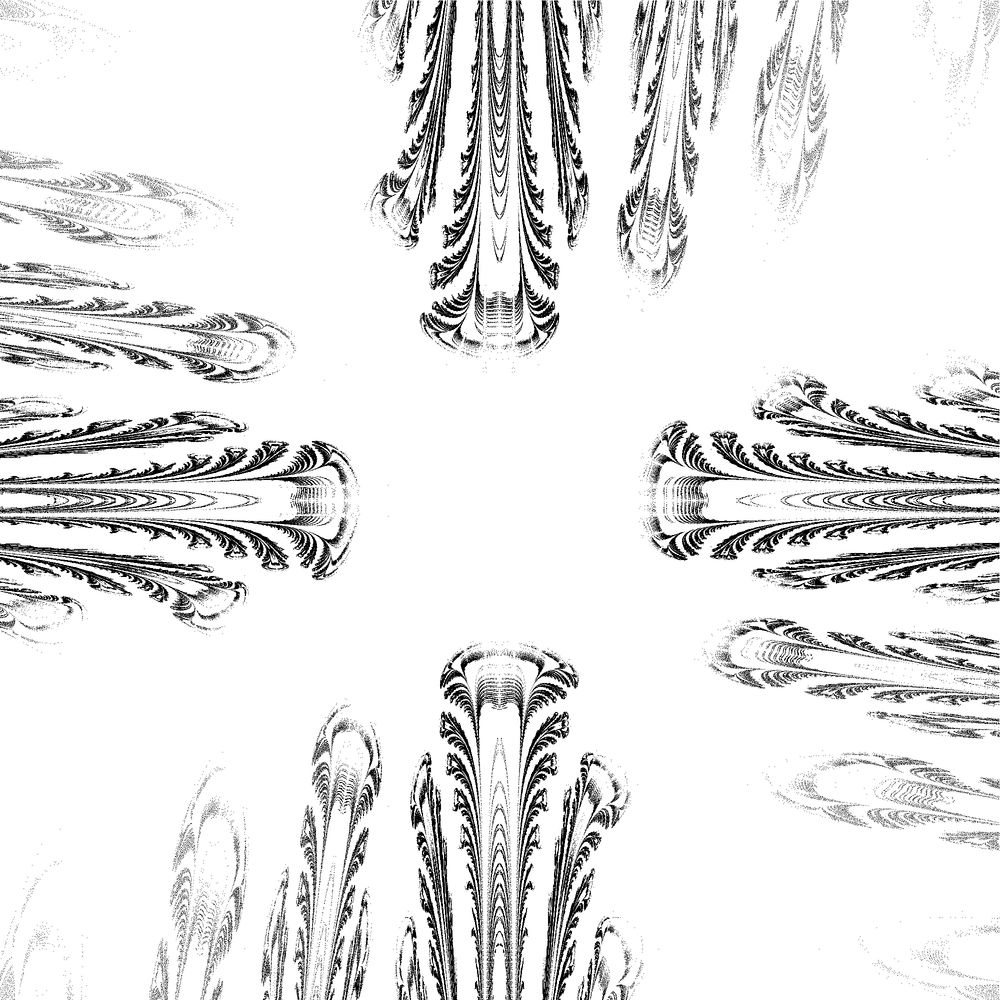}\\
  \end{tabular}
  \caption{%
    \em\small
    $\omega$-limits (in color, each color corresponding to a different basin of attraction) and $\alpha$-limits (in black and white)
    for the Newton maps of  the polynomial maps (from top to bottom) $f(x,y)=(10x(x^2-1)+3y,y(y^2-1)-x)$,
    $g(x,y)=(10x(x^2-1)+7y,y(y^2-1)-x)$ and $h(x,y)=(x(x^2-1)+60y,y(y^2-1)-60x)$.
  }
  \label{fig:c3}
  \end{figure}
\begin{figure}
  \centering
  \begin{tabular}{cc}                                                                                                                                                                  
    \includegraphics[width=5.6cm]{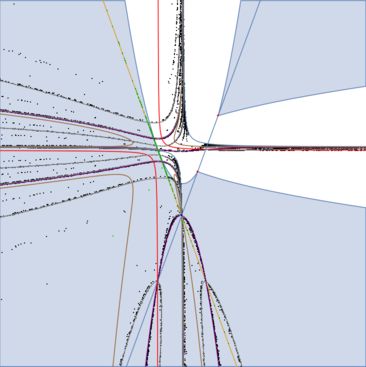}&\includegraphics[width=5.6cm]{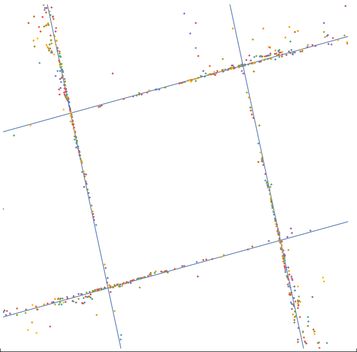}\\
    \includegraphics[width=5.6cm]{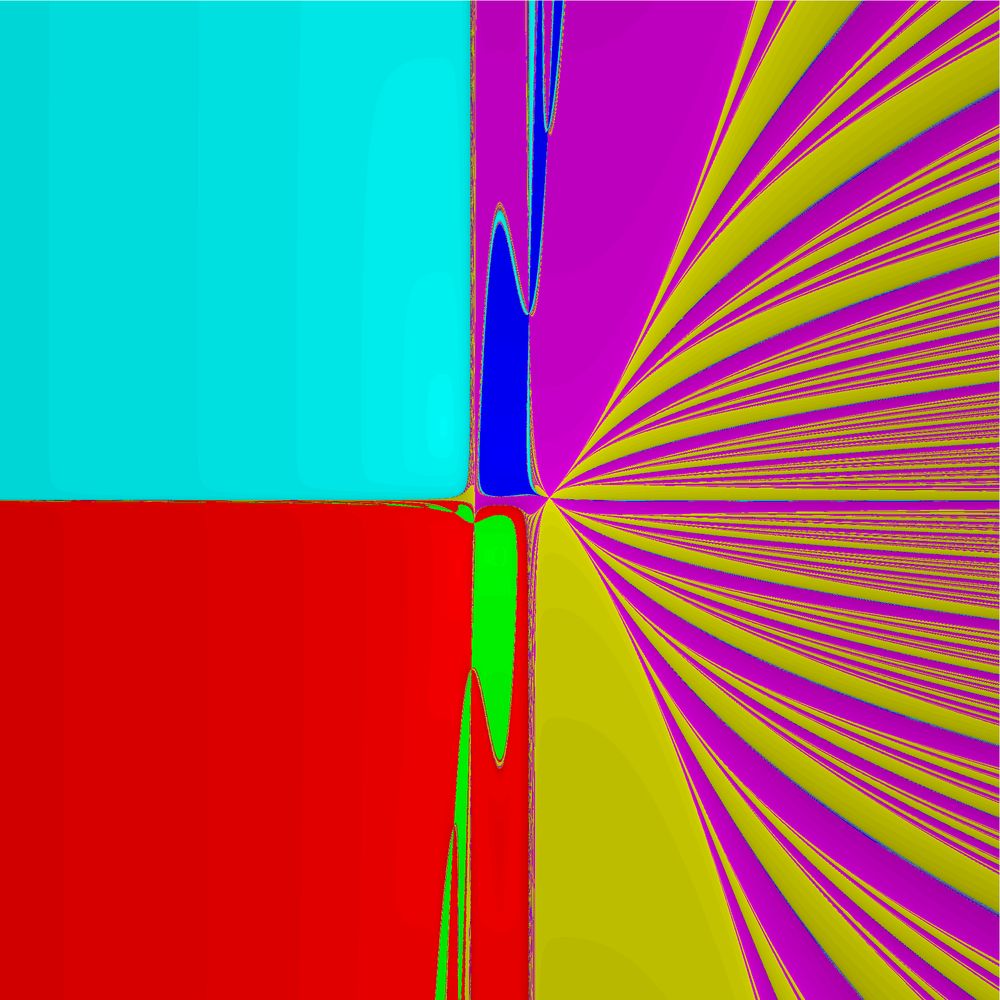}&\includegraphics[width=5.6cm]{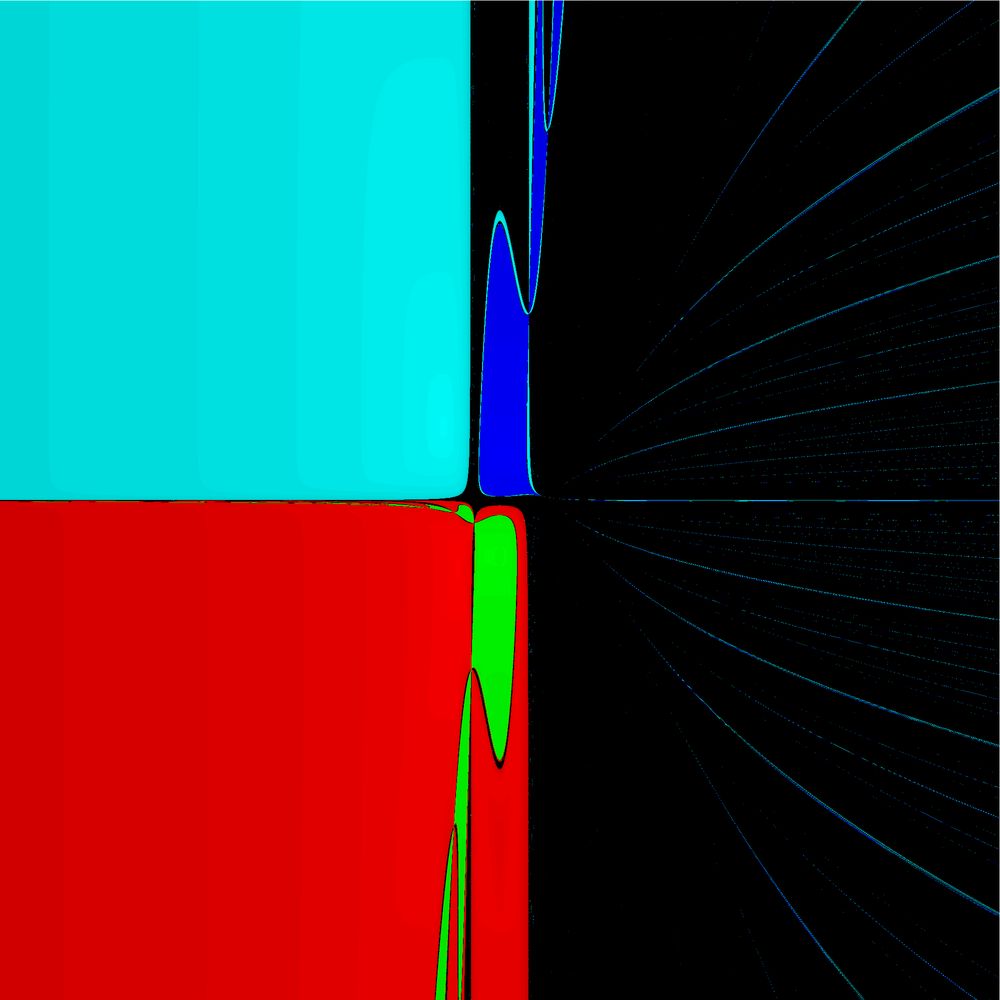}\\
    \includegraphics[width=5.6cm]{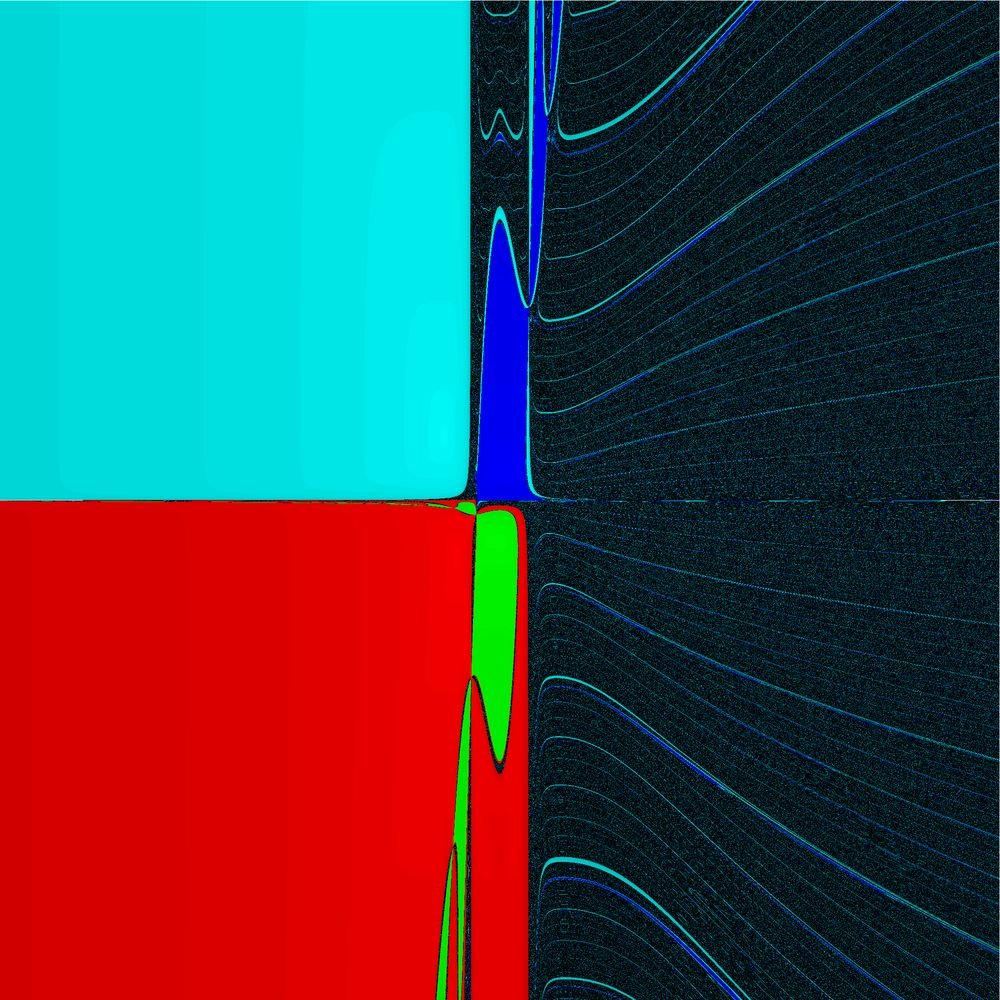}&\includegraphics[width=5.6cm]{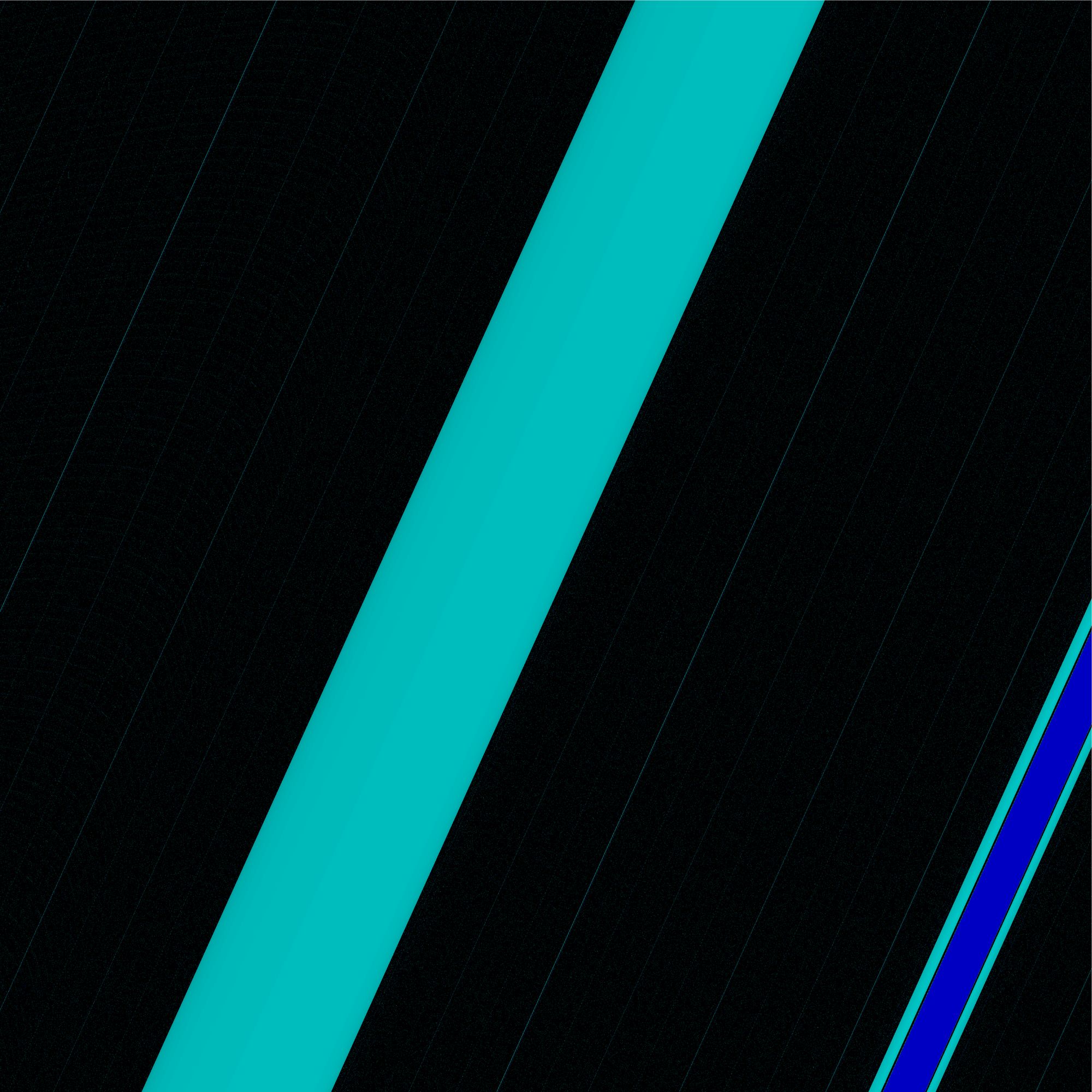}\\
  \end{tabular}
  \caption{%
    \em\small\vspace{-.1cm}
    (Top, left) Main elements of the dynamics of $N_g$ in Fig.~\ref{fig:q1}. (Top, right) A typical orbit of a cyan point in
    Fig.~\ref{fig:c3}(middle) and the 4 ghost lines of the map. (Middle and bottom) Basins of attraction of $N_{f_\alpha}$ for
    $f_\alpha(x,y)=(x^2(x-1)+y,x-\alpha-y^2)$ with resp. $\alpha=-0.997462,-0.997461,-0.5$ in the square $[-10,10]^2$
    and a detail in $[3.00000015,3.00000025]\times[4,4.0000001]$ for $\alpha=-0.995$.
%    , $[-10,10]^2$, $[3.00000015,3.00000025]\times[4,4.0000001]$, $[-10,10]^2$.
%    (Top, left) Main elements of the dynamics of $N_g$ in Fig.~\ref{fig:q1}: fixed points (red), invariant line through the two roots (blue),
%    invariant ghost line (light brown), hyperbola of points sent to infinity (blue) with its 1st (red) and 2nd (brown) counterimages
%    and the first 500 points (green) of the orbit of a random cyan point in Fig.~\ref{fig:q1}. The white points in the background
%    are those not belonging to $N_g(\Rt)$. (Top, right) A typical orbit of a cyan point in Fig.~\ref{fig:c3}(middle) and the 4 ghost lines
%    of the map. (Middle and bottom) Basins of attraction of the Newton map of $f_\alpha(x,y)=(x^2(x-1)+y,x-\alpha-y^2)$ for,
%    respectively, $\alpha=-0.997462,-0.997461,-0.995,-0.5$ in the squares $[-10,10]^2$, $[-10,10]^2$,
%    $[3.00000015,3.00000025]\times[4,4.0000001]$.
  }
  \label{fig:misc}
  \end{figure}

Our final observation is that, while a few examples of Newton maps of polynomials with maximal number of roots
studied numerically cannot grant that our Conjecture 2 holds, one thing that supports its validity is how easy it is
to find such conditions violated as soon as the number of real roots is not maximal anymore. 
\section*{Acknowledgments}
I am in great debt with J. Yorke and J. Hawkins for many discussions that greatly helped the development of this article
and I am grateful to S. van Strien for some clarification on real one-dimensional dynamics. All pictures, except the two in
the top row of Fig.~\ref{fig:misc}, were generated by code written by the author in Python and C/C++. All calculations were
performed on the HPCC of the College of Arts and Sciences at Howard University.
\bibliography{refs}
\end{document}